\documentclass{amsart}

\usepackage{amssymb,tikz,url,stmaryrd}
\usepackage[pagebackref=true]{hyperref}
\usetikzlibrary{shapes,backgrounds,matrix,arrows,automata,positioning}
\usepackage[all]{xy}

\def\Dd{\mathcal{D}}

\def\Jj{\mathcal{J}}
\def\Rr{\mathcal{R}}
\def\Ll{\mathcal{L}}
\def\Nn{\mathcal{N}}
\def\Pp{\mathcal{P}}
\def\Xx{\mathfrak{X}}
\def\Yy{\mathfrak{Y}}
\def\Set{\mathrm{Set}}
\def\op{\mathrm{op}}
\def\el{\mathrm{el}}
\def\KO{\mathrm{KO}}
\def\Uu{\mathcal{U}}
\def\supp{\mathrm{supp}}
\def\Frm{\mathrm{Frm}}
\def\Hom{\mathrm{Hom}}
\def\CC{\mathbb{C}}

\def\DD{\mathbb{D}}
\def\one{\mathbf{1}}
\def\Ff{\mathcal{F}}
\def\Gg{\mathcal{G}}

\def\PM{T}

\def\LRB{\mathrm{LRB}}
\def\RRB{\mathrm{RRB}}
\def\Poset{\mathrm{Poset}}

\def\cm{\curlywedge}
\def\cj{\curlyvee}
\def\Spec{\mathrm{Spec}}
\def\Sh{\mathrm{Sh}}
\def\colim{\mathrm{colim}}

\def\Sets{\mathrm{Sets}}
\def\id{\mathrm{id}}

\def\veedot{\ ^\cdot \hspace{-7.2pt}\vee}
\def\sveedot{\ ^\cdot \hspace{-5.2pt}\vee}
\def\bigveedot{\ ^\cdot \hspace{-9.6pt}\bigvee}
\def\inlinebigveedot{\ ^\cdot \hspace{-7.6pt}\bigvee}

\def\eqv{\overset{\sim}{\longrightarrow}}
\def\lra{\leftrightarrows}
\def\dto{\rightrightarrows}
\def\onto{\twoheadrightarrow}
\def\into{\hookrightarrow}

\theoremstyle{plain}
\newtheorem{thm}{Theorem}[section]
\newtheorem{dfn}[thm]{Definition}
\newtheorem{lma}[thm]{Lemma}
\newtheorem{prp}[thm]{Proposition}
\newtheorem{cor}[thm]{Corollary}

\theoremstyle{remark}

\newtheorem{rmk}[thm]{Remark}
\newtheorem{exm}[thm]{Example}

\title[Stone duality for spectral sheaves]{Stone duality for spectral sheaves\\ and the patch monad
}

\thanks{This project has been funded by the European Research Council (ERC) under the European Union's Horizon 2020 research and innovation program (grant agreement No.670624). }

\author{Clemens Berger}

\address{Universit\'e C\^ote d'Azur, Lab. J. A. Dieudonn\'e, Parc Valrose, 06108 Nice Cedex, France}
\email{cberger@math.unice.fr}

\author{Mai Gehrke}

\address{Universit\'e C\^ote d'Azur, Lab. J. A. Dieudonn\'e, Parc Valrose, 06108 Nice Cedex, France}
\email{mai.gehrke@unice.fr}


\date{February 2, 2022}


\begin{document}
%%%%%%%%%%%%%%%%%%%%

%%%%%%%%%%%%%%%%%%%%
\begin{abstract}
We establish a duality between global sheaves on spectral spaces and right distributive bands. This is a sheaf-theoretical extension of classical Stone duality between spectral spaces and bounded distributive lattices.

The topology of a spectral space admits a refinement, the so-called patch topology, giving rise to a patch monad on sheaves over a fixed spectral space. Under the duality just mentioned the algebras of this patch monad are shown to correspond to distributive skew lattices.
\end{abstract}
%%%%%%%%%%%%%%%%%%%%

\maketitle

%%%%%%%%%%%%%%%%%%%%
\section*{Introduction}
%%%%%%%%%%%%%%%%%%%%
Classical Stone duality \cite{Stone37} between spectral spaces and bounded distributive lattices has been inspirational for several domains such as algebraic geometry \cite{Groth}, categorical logic \cite{Johnstone} and theoretical computer science \cite{Abr91}. In this article we deal with a sheaf-theoretical extension of Stone duality. Starting point is the observation that the local sections of any sheaf come equipped with a pertinent non-commutative algebraic structure. We call this structure a \emph{distributive band} because its universal commutative quotient is a \emph{distributive lattice}. If the sheaf is defined on a spectral space, the distributive band of local sections over compact open subsets determines the sheaf as well as the space. This results in a dual equivalence between \emph{global sheaves on spectral spaces} and \emph{right distributive bands}, cf. Theorem \ref{thm:main0}.

Spectral spaces which satisfy Hausdorff's separation axiom are usually called \emph{Stone spaces} or \emph{Boolean spaces}. Global sheaves on Boolean spaces have attracted much attention, and, restricting to these, we recover known dualities, cf. Bergman \cite{Berg} and Kudryavtseva-Lawson \cite{KL}. The benefit of a duality over general spectral spaces is its interaction with the so-called \emph{patch topology}, cf. \cite{Groth,Hoch}. Replacing the topology of a spectral space $\Xx$ with the finer patch topology defines a Boolean space $\Xx^p$, which is also known in literature as the \emph{Priestley space} \cite{P} associated with the spectral space provided the \emph{specialisation order} of $\Xx$ is kept as part of the structure of $\Xx^p$. The identity mapping $\phi:\Xx^p\to\Xx$ induces then a monad $T=\phi_*\phi^*$ on the category of sheaves on $\Xx$ which we call the \emph{patch monad}. This construction is reminiscent of Godement's flabbyfication monad \cite{God} for general sheaves, based on the discretisation of the topology. We obtain two main results concerning the structure of $T$-algebras, cf. Theorems \ref{thm:main1} and \ref{thm:main2}:
\begin{enumerate}
\item a global sheaf on $\Xx$ carries a $T$-algebra structure if and only if the right distributive band of local sections is actually a \emph{right distributive skew lattice} in the sense of Leech \cite{Le3}, cf. Section \ref{subsct:skew} for terminology. These $T$-structured spectral sheaves are in particular quasi-flasque \cite{Kempf}: all restriction maps between compact open subsets are surjective, cf. Proposition \ref{prp:skewsheaf}.

\item the category of $T$-algebras is equivalent to the category $\Sh_{s}(\Xx^p)$ of sheaves on $\Xx^p$ with \emph{saturated} support (i.e. the support is an upset for the Priestley order). More precisely, the restriction of the right adjoint functor $\phi_*:\Sh(\Xx^p)\to\Sh(\Xx)$ to sheaves with saturated support is \emph{monadic} and induces thus an equivalence of categories $\Sh_{s}(\Xx^p)\eqv\Sh(\Xx)^{\phi_*\phi^*}$.
\end{enumerate}

The structural characterisation (1) identifies the notion of distributive skew lattice as natural from a sheaf-theoretical point of view. Together with monadicity (2) this recovers the dual equivalence between the categories of distributive skew lattices and of globally supported Priestley sheaves, obtained by Bauer et al. in \cite{BCGGK}. Another outcome of monadicity is the observation that the category of $T$-algebras is an exact category in the sense of Barr \cite{Barr}. Finally, we deduce from (1) and (2) that distributive skew lattices admit a \emph{Boolean skew lattice envelope}, generalising the Boolean envelope of a bounded distributive lattice, cf. Proposition \ref{prp:Bool2}.

The results of this article relate concepts from different areas in mathematics with usually few interactions. The new perspective we propose is to consider \emph{sheaf representations} for ordered structures rather than for algebraic structures, where nonetheless we rely on \emph{monad theory} to incorporate change of topology. We have tried to supply enough details to allow interested readers to follow the main stream of ideas no matter to which of these areas they belong.  %As guideline  for semigroup theory we suggest Howie's book \cite{Howie},  for sheaf theory Moerdijk-Mac Lane's book \cite{MoerMacL} and for monad theory Barr-Wells' lecture notes \cite{BarrWells}. A recent reference for spectral spaces, going much beyond of what we need here, is the book \cite{DST} by Dickmann-Schwartz-Tressl.\vspace{1ex}

The article is subdivided into four sections, the first two being mostly expository:

Section \ref{sct:bands} reviews idempotent semigroups with emphasis on regular and normal bands. We characterise right normal bands among right regular bands using a comprehensive factorisation system \cite{BK,SW}. This sheds light on an important theorem of Kimura-Yamada \cite{Ki,YK}: every right normal band may be represented as the category of elements of a canonical presheaf on its universal semilattice quotient.

Section \ref{sct:skew} introduces distributive bands and relates them to distributive skew lattices in the sense of Leech \cite{Le1}. We review Leech's equational characterisation \cite{Le3} of distributive skew lattices among normal skew lattices using a new characterisation of symmetric skew lattices among normal skew lattices.

Section \ref{sct:duality} contains our first main result: a dual equivalence between the category of global sheaves on spectral spaces and the category of right distributive bands. When restricted to terminal sheaves this recovers classical Stone duality \cite{Stone37}.

Section \ref{sct:patch} investigates the patch monad. Its algebras are related to distributive skew lattices on one side, and to Priestley sheaves with saturated support on the other. We give \emph{two} proofs of the equivalence with saturated Priestley sheaves: an ``abstract'' proof based on Duskin's monadicity criterion \cite{D}, and a ``constructive'' proof exhibiting a quasi-inverse to the Eilenberg-Moore comparison functor.

Throughout the text we denote bands by $X,Y,$ spectral spaces by $\Xx,\Yy,$ their compact open subsets by $U,V,$ and their dual lattices by $\KO(\Xx),\KO(\Yy)$. Sheaves will be denoted by $\Ff,\Gg,$ the local sections of a sheaf by $s,t$. A sheaf is called global (resp. globally supported) if it has global sections (resp. global support).

\subsection*{Acknowledgements} We are grateful to Maria Manuel Clementino, Richard Garner and Steve Lack for useful discussions concerning Duskin's monadicity theorem.

%%%%%%%%%%%%%%%%%%%%
\section{Regular and normal bands}\label{sct:bands}
%%%%%%%%%%%%%%%%%%%%

Idempotent semigroups are usually called \emph{bands}. We are mainly interested in a special class of normal bands which we call \emph{distributive} because of their close relationship with distributive skew lattices in the sense of Leech \cite{Le1}. Our definition of distributive band, to be presented in Section \ref{sct:skew}, relies on Kimura's representation theorem for \emph{normal bands} \cite{Ki}. For the reader's convenience we provide here a self-contained introduction to regular and normal bands, respectively.

We will often restrict ourselves to \emph{right regular bands}. This choice is a matter of convention; left regular bands would do as well. The main novelty of this introductory section is that we relate Kimura's representation theorem for normal bands to the existence of an orthogonal factorisation system for maps of right regular bands, thus placing Kimura's result in a wider perspective. This is done by lifting the \emph{comprehensive factorisation} system of Street and Walters \cite{SW,BK} from partially ordered sets to right regular bands showing that any morphism of right regular bands factors essentially uniquely as a connected map followed by a covering. The \emph{normal bands} are then identified with those regular bands whose semilattice reflection is a covering. In particular, any right normal band can be realised as the \emph{category of elements} of the presheaf given by the covering of its universal semilattice quotient.
%%%%%%%%%%%%%%%%%%%%
\subsection{Preliminaries on bands}
%%%%%%%%%%%%%%%%%%%%
\begin{dfn}
Let $X$ be a band. We define a relation $\leq$ on $X$ by
\[
x\leq y\quad \iff\quad x=yxy \quad \iff\quad x=yx=xy.
\]
Further, we define a relation $\preceq$ on $X$ by
\[
x\preceq y\quad \iff\quad x=xyx.
\]
Let $\Dd$ be the symmetric part of $\preceq$. That is, the relation given by
\[
x\Dd y   \quad \iff\quad    x\preceq y\text{ and }y\preceq x\iff x=xyx\text{ and }y=yxy.
\]
\end{dfn}

\begin{lma}\label{lma:idemGreen}Let $X$ be a band and $x,y,z\in X$.
\begin{itemize}\item[(i)]The relation $\leq$ is a partial order on $X$. It is preserved by semigroup homomorphisms and  a two-sided $0$ for $X$ is the same as a bottom of $(X,\leq)$.
\item[(ii)]The relation $x\preceq y$ holds if and only if there exist $s,t\in X$ with $x=syt$;\item[(iii)]The relation $x\preceq y$ is a preorder such that $xz\preceq yz$ and $zx\preceq zy$.\end{itemize}
\end{lma}

\begin{proof}
(i) It is clear from the definition of $\leq$ that it is reflexive and antisymmetric, that it is preserved by semigroup homomorphisms and that an element is the two-sided $0$ of $X$ if and only if it is the bottom of $(X,\leq)$. Finally, if $x\leq y$ and $y\leq z$ then $x=xy=yx$ and $y=zy=yz$, and thus $zx=z(yx)=(zy)x=yx=x$ and $xz=(xy)z=x(yz)=xy=x$ so that $x\leq z$ and $\leq$ is transitive.

(ii) If $x\preceq y$ then $x=xyx$ and thus we may take $s=t=x$. Conversely, suppose $x=syt$. Then we have
\[
xyt=(syt)yt=s (yt)(yt)=syt=x
\]
and thus
\[
xyx=xy(xyt)=(xy)(xy) t=xyt=x.
\]

(iii) By idempotence of the band operation the relation $\preceq$ is reflexive and by (ii) it is transitive. Thus $\preceq$ is a preorder. But $x\preceq y$ means $x=xyx$ so that for any $z\in X$ we have $xz=xyxz=(x) y (xz)$. By (ii) we get $xz\preceq y$ and thus
\[
xz=(xz)y(xz)=x(zy)(xz)=x(zy)(zy)(xz)=(xz) yz (yxz)
\]
whence $xz\preceq yz$ again by (ii). The relation $zx\preceq zy$ follows symmetrically.\end{proof}

\begin{rmk}
The preorder $\preceq$ is usually known as the $\Jj$-preorder and its symmetric part is known as Green's $\Jj$-relation. In any idempotent semigroup Green's $\Jj$- and  $\Dd$-relations coincide -- whence our notation.
\end{rmk}

We recall a few standard definitions. A commutative band is called a \emph{semilattice}. A semilattice carries a partial order given by $x\leq y$ if and only if $x=xy=yx$. For arbitrary elements $x$ and $y$ in the semilattice, the infimum of $x$ and $y$ for this partial order is given by $xy=yx$. Accordingly we will denote the semigroup operation in a semilattice by $\wedge$. A band is called \emph{rectangular} provided $X/\Dd$ is singleton or empty.
A subset $S$ of a partially ordered set is said to be \emph{order-discrete} provided the restricted partial order on $S$ is discrete, \emph{i.e.} for all $x,y\in S$, if $x\leq y$ then $x=y$.

The following proposition collects a number of well-known facts about semigroups. We include proofs for readers not familiar with semigroup theory.

\begin{prp}\label{prp:basic}For any band $X$ the following properties hold:\begin{itemize}
\item[(i)] The $\Dd$-relation is a semigroup congruence on $X$;
\item[(ii)]The quotient $X/\Dd$ is the universal semilattice quotient of $X$;
\item[(iii)]For $x,y\in X$, if $x\leq y$ then $x\preceq y$, and if $x\leq y\text{ and }y\preceq x$ then $x=y$.
\item[(iv)]$(X,\leq)$ is a partially ordered set in which $\Dd$-classes are order-discrete;
\item[(v)] The quotient map $(X,\leq)\to(X/\Dd,\leq)$ is order preserving. More precisely,
\[
x\leq y\implies x\preceq y\iff[x]_\Dd\preceq[y]_\Dd\iff[x]_\Dd\leq[y]_\Dd.
\]
\item[(vi)]$X$ is rectangular if and only if $(X,\leq)$ is order-discrete.
\end{itemize}\end{prp}

\begin{proof}
(i) By Lemma~\ref{lma:idemGreen}ii, the $\Dd$-relation is the symmetric part of the compatible preorder $\preceq$ and hence a semigroup congruence.

(ii) According to (i) the quotient map $X\to X/\Dd$ is a surjective semigroup homomorphism. The quotient band is commutative since $xy\Dd yx$ by
\[xy=xyxy=xy(yx)xy\text{ and }yx=yxyx=yx(xy)yx.\]
Thus $X/\Dd$ is a semilattice. All that remains to be shown is the universal property. To this end, let $h\colon X \to L$ be a map of bands into a semilattice and let $x,y\in X$ with $x\Dd y$. Then $h(x)=h(xyx)=h(x)\wedge h(y)$ and thus $h(x)\leq h(y)$. Similarly $y=yxy$ implies $h(y)\leq h(x)$ so that $h(x)=h(y)$. Therefore $\Dd$ is contained in the kernel of $h$, and $h$ factors through the quotient map $X\to X/\Dd$ as required.

(iii) If $x\leq y$ then $x=xy$ and hence $x=xx=xyx$ so that $x\preceq y$. Further, if $x\leq y$ and $y\preceq x$ then $x=yxy$ and $y=yxy$ so that $x=y$.

(iv) According to (iii) the relation $\leq$ is antisymmetric. Idempotence of the band operation implies reflexivity. Using that $x\leq y$ if and only if $x=xy=yx$, and that $y\leq z$ if and only if $y=yz=zy$, we get $x=xy=x(yz)=(xy)z=xz$ and $x=yx=(zy)x=z(yx)=zx$ thus $x\leq z$, that is transitivity.

It follows from (iii) that $\Dd$-classes are order-discrete.

(v) The left and middle implications from left to right follow from (iii) and (i).

The rightmost equivalence follows from commutativity of the quotient $X/\Dd$.

If $[x]\preceq[y]$ then there exist $x',y'$ such that $x\Dd x'$ and $y\Dd y'$ and $x'\preceq y'$. This implies $x\preceq y$ by transitivity of $\preceq$ and the fact that $\Dd$ is the symmetric part of $\preceq$. Therefore the middle implication is an equivalence too.

(vi) If $X$ is rectangular then $X$ is either empty or consists of a single $\Dd$-class which is order-discrete by (iv). If $(X,\leq)$ is order-discrete then the identity $x(xyx)x=xyx$ implies $xyx\leq x$ and hence $xyx=x$ so that $X$ is rectangular.\end{proof}

Property (ii) says that commutative bands form a \emph{full reflective subcategory} of the category of all bands. Commutative bands are \emph{meet-semilattices} with respect to the partial order $\leq$. However, one should be careful: though the reflection $X\to X/\Dd$ takes the product $xy$ to the meet $[x]\wedge[y]$ in $X/\Dd$, it is possible that the meet $x\wedge y$ in $(X,\leq)$ exists but nevertheless $x\wedge y\not=xy$, cf. Lemma~\ref{lma:commuting} below.
%%%%%%%%%%%%%%%%%%%%
\subsection{Regular bands}
%%%%%%%%%%%%%%%%%%%%

Let $X$ be a band and $z\in X$. Note that the order-ideal generated by $z$ is given by
\[
{\downarrow} z=\{x\in X\mid x\leq z\}=zXz.
\]
This description shows that ${\downarrow} z$ is a subband of $X$ and that we have a canonical `projection map'  $p_z\colon X\to {\downarrow} z$ given by $p_z(x)=zxz$. These projection maps are not in general semigroup homomorphisms, but a band $X$ is called \emph{regular} provided the projections $x\mapsto zxz$ onto the order-ideals ${\downarrow}z=zXz$ are homomorphisms for all $z$, i.e. for all $x,y,z\in X$ we have $zxzyz=zxyz$.
A band $X$ is called \emph{right} (resp. \emph{left}) \emph{regular} if $xyx=yx$ (resp. $xy=xyx$) for all $x,y$ in $X$.

For any semigroup $X$, Green's $\Ll$- and $\Rr$-relations are defined by $x\Ll y$ if and only if $Xx=Xy$, resp. $x\Rr y$ if and only if $xX=xY$. By the same argument as in the proof of Lemma~\ref{lma:idemGreen}, in case $X$ is a band, the $\Ll$-relation amounts to the conjunction of $xy=x$ and $yx=y$, and the $\Rr$-relation to the conjunction of $xy=y$ and $yx=x$. As for any semigroup, the $\Ll$-relation is always a right congruence and the $\Rr$-relation a left congruence. It is not too difficult to show that a band $X$ is regular if and only if $\Ll$ and $\Rr$ are both semigroup congruences. The quotient band $X/\Ll$ is actually the universal \emph{right regular} reflection of $X$, while the quotient band $X/\Rr$ is the universal \emph{left regular} reflection of $X$.

Further, since $\Ll$ and $\Rr$ are contained in $\Dd$, we get a commutative diagram
  $$\xymatrix{X\ar[r]\ar[d] &X/\Ll\ar[d]\\X/\Rr\ar[r] & X/\Dd}$$

For any band, the product map $X\to X/\Ll\times X/\Rr$ is injective and for any $x,y$ with $x\Dd y$, one has $x\Ll yx$ and $y\Rr yx$ so that $([y]_\Rr,[x]_\Ll)=([yx]_\Rr,[yx]_\Ll)$. The square above is thus a pullback square of underlying sets. In particular, if $X$ is a regular band then the square is a pullback square of semigroups, i.e. $X\cong X/\Rr\times_{X/\Dd}X/\Ll$ in the category of semigroups. This important fact was originally proved by Kimura \cite{Ki} and often allows one to obtain general results for regular bands by establishing them just for left (resp. right) regular bands, thus simplifying arguments.

Since a \emph{rectangular} band $X$ is a regular band with a single $\Dd$-class, the pullback above produces for each rectangular band $X$ a canonical decomposition $$X\cong X/\Rr\times X/\Ll$$
into the product of a \emph{left rectangular} band $X/\Rr$ (i.e. $[x]_\Rr[y]_\Rr=[x]_\Rr$) and a \emph{right rectangular} band $X/\Ll$ (i.e. $[x]_\Ll[y]_\Ll=[y]_\Ll$). Since the binary operation of a left (resp. right) rectangular band is uniquely determined, the same is true for any rectangular band.

From hereon out, we shall restrict our attention to bands that are (at least) regular. In fact, since $X\cong X/\Rr\times_{X/\Dd}X/\Ll$ for any such band, we will mainly focus on the categories $\LRB$ of left regular bands and $\RRB$ of right regular bands, respectively. We will need the following general lemma about regular bands.

\begin{lma}\label{lma:commuting}If two elements $x,y$ of a regular band $X$ commute then the product $xy$ represents the meet $x\wedge y$ in $(X,\leq)$. Conversely, if there exists $z\in X$ such that $z\leq x$ and $z\leq y$ and $[z]_\Dd=[x]_\Dd\wedge[y]_\Dd$ then $xy=z=yx$.\end{lma}

\begin{proof}If $xy=yx$ then $xy\leq x$ and $xy\leq y$ and for $z\in X$ such that $z\leq x$ and $z\leq y$ we get $z=yxzxy=xyzxy$ thus $z\leq xy$. Conversely, if such a $z\in X$ satisfies $[z]_\Dd=[x]_\Dd\wedge[y]_\Dd=[xy]_\Dd=[yx]_\Dd$ then Proposition~\ref{prp:basic}iii implies $xy=z=yx$.\end{proof}

%%%%%%%%%%%%%%%%%%%%
\subsection{Comprehensive factorisation}
\label{sct:comprehension}
%%%%%%%%%%%%%%%%%%%%

Street-Walters \cite{SW} construct a comprehensive factorisation of any functor between small categories, see also Par\'e \cite{Par}. On the other hand, in \cite{BK} a general theory of comprehensive factorisation systems is established, inspired by the fact that any continuous map between locally simply connected topological spaces admits a formally analogous factorisation based on topological covering theory. As in \cite{BK}, our terminology is inspired by this example.

\emph{Partially ordered sets} (posets for short) can be considered as small categories in which relations $x\leq y$ are morphisms with domain $x$ and codomain $y$. Reflexivity of the order relation yields the identity morphisms, transitivity the compositions, and antisymmetry amounts to the property that the only invertible morphisms of the category are the identities. We will tacitly identify posets with the corresponding small categories, which we will refer to as posetal categories. Note also that the functors between such posetal categories are precisely the order preserving maps. Our purpose here is to show that Street-Walters' comprehensive factorisation of an order preserving map ``lifts'' to a comprehensive factorisation of a map of left (resp. right) regular bands. Let us make this statement more precise.

\begin{dfn}
A map of posets $f:(X,\leq)\to(Y,\leq)$ is said to be a \emph{covering} (aka a \emph{discrete fibration}) if for each $x\in X$ and $y'\leq f(x)$ there exists a unique $x'\leq x$ such that $f(x')=y'$.
\end{dfn}
\begin{rmk}\label{cover=elementproj}
The best way to understand covering maps is on the basis of \emph{presheaves}. In fact, any covering $f:X\to Y$ is isomorphic (over $Y$) to the projection map $\el_Y(\Ff)\to Y$ where the source $\el_Y(\Ff)$ is the poset of elements of an, up to isomorphism, uniquely determined set-valued presheaf $\Ff:(Y,\leq)^\op\to\Set$.

Given a set-valued presheaf $F:(Y,\leq)^\op\to\Set$ the poset $\el_Y(F)$ is the poset whose elements are the pairs $(y,s)$ consisting of an element $y\in Y$ and an element $s\in F(y)$ and with the partial order $(y,s)\leq(z,t)$ precisely when $y\leq z$ and $F(y\leq z)$ takes $t$ to $s$. Further, the projection $\pi:\el_Y(F)\to Y$ taking $(y,s)$ to $y$ is a covering. Conversely, for a given covering $f:(X,\leq)\to(Y,\leq)$ the \emph{representing presheaf} $F_f$ is defined by $F_f(y)=f^{-1}(y)$ with obvious action maps $F_f(y'\leq y):f^{-1}(y)\to f^{-1}(y')$ given by the covering property.
\end{rmk}

We denote by $PX$ the category of set-valued presheaves on $(X,\leq)$. A map of posets $f:(X,\leq)\to(Y,\leq)$ is \emph{connected} (aka \emph{final}) provided (cf. \cite{BK,SW}) the \emph{left Kan extension} along $f:(X,\leq)\to(Y,\leq)$ takes a terminal presheaf $\star_{PX}$ on $(X,\leq)$ to a terminal presheaf $f_!(\star_{PX})$ on $(Y,\leq)$. A terminal presheaf $\star_{PX}$ sends each element of $X$ to a singleton set with the obvious restriction maps. Also, one may show that the set-valued presheaf $f_!(\star_{PX})$ assigns to $y$ in $Y$ the set $\pi_0(y{\downarrow} f)$ of connected components of the poset $y{\downarrow} f$ where the posetal comma category $y{\downarrow} f$ has as objects those $x\in X$ for which $y\leq f(x)$ with partial order induced by $(X,\leq)$.

Having unraveled the definitions, we see that a map of posets $f:(X,\leq)\to(Y,\leq)$ is \emph{connected} if and only if, for each $y\in Y$, the poset $y{\downarrow} f=\{x\in X\,|\,y\leq f(x)\}$ is a non-empty connected subposet of $(X,\leq)$.

The \emph{comprehensive factorisation} of an order preserving map $f:(X,\leq)\to(Y,\leq)$ can now be described  explicitly as
\[
X\overset{\alpha_f}{\longrightarrow}\el_Y(F)\overset{\beta_f}{\longrightarrow}Y.
\]
Here $F=f_!(\star_{PX})$ and $\alpha_f(x)=(f(x),[f(x)]_{f(x){\downarrow} f})$ where $[f(x)]_{f(x){\downarrow} f}$ denotes the connected component of $f(x)$ inside the poset $f(x){\downarrow} f$. The map $\beta_f$ is the projection given by $\beta_f((y,s))=y$.

This defines an \emph{orthogonal} factorisation system for the category of posets with left part consisting of the connected maps and right part consisting of the coverings, cf. \cite{SW,BK}. We shall now show that this factorisation system ``lifts'' along the \emph{forgetful} functor $U:\LRB\to\Poset$ (resp. $U:\RRB\to\Poset$) taking a left (resp. right) regular band $X$ to its underlying poset $UX=(X,\leq)$ where $x\leq x'$ if and only if $x=x'xx'$, cf. Section~\ref{sct:bands}. Such a ``lifting'' means that for each map of left (resp. right) regular bands $f:X\to Y$ there exists an essentially unique factorisation $f=\beta_f\alpha_f$ such that $U\alpha_f=\alpha_{Uf}$ and $U\beta_f=\beta_{Uf}$.

\begin{lma}\label{lma:elements}For any left (resp. right) regular band $Y$ and set-valued presheaf $F:(Y,\leq)^\op\to\Set$ the poset of elements $\el_Y(F)$ underlies a unique structure of left (resp. right) regular band such that the projection map $\el_Y(F)\to Y$ is a band homomorphism. The induced semilattice map $\el_Y(F)/\Dd\hookrightarrow Y/\Dd$ is injective and identifies the domain with a downset of the codomain.\end{lma}

\begin{proof}We assume that $Y$ is a right regular band, the proof for left regular bands is dual. The poset of elements $\el_Y(F)$ has as objects pairs $(x,s)$ consisting of an element $x\in Y$ and an element $s\in F(x)$ with partial order $(x,s)\leq(y,t)$ precisely when $x\leq y$ and $F(x\leq y)$ takes $t$ to $s$. We write $s=t\rceil^y_x$. Since $Y$ is right regular, we have $xy\leq y$ so that the covering property of the projection map $\el_Y(F)\to Y$ forces us to put $(x,s)(y,t)=(xy,t\rceil^y_{xy})$. This defines indeed a right regular band structure on $\el_F(Y)$ since $$(x,s)(y,t)(x,s)=(xyx,s\rceil^x_{xyx})=(yx,s\rceil^x_{yx})=(y,t)(x,s).$$ Moreover $(x,s)\leq(y,t)$ if and only if $(x,s)(y,t)=(x,s)$ so that the partial order on $\el_Y(F)$ underlies the right regular band structure.
For right regular bands Green's $\Dd$-relation reads\begin{align*}(x,s)\Dd(y,t)&\iff(x,s)(y,t)=(y,t)\text{ and }(y,t)(x,s)=(x,s)\\&\iff xy=y\text{ and }yx=x\text{ and }(s,t)\in F(x)\times F(y)\\&\iff x\Dd y\text{ and }(s,t)\in F(x)\times F(y)\end{align*}so that the induced map $\el_Y(F)/\Dd\to Y/\Dd$ is injective. Finally, notice that the image consists of those $[y]_\Dd$ for which $F(y)\neq\emptyset$ and these form a downset. \end{proof}

\begin{dfn}\label{defcovering}A map of regular bands $f:X\to Y$ is said to be a \emph{covering} (resp. \emph{connected}) if the underlying map of posets $Uf:(X,\leq)\to(Y,\leq)$ is such.\end{dfn}

\begin{prp}\label{prp:factorisation}Any map of left (resp. right) regular bands factors essentially uniquely as a connected map followed by a covering. The resulting comprehensive factorisation system is orthogonal and stable under pullback along coverings.\end{prp}

\begin{proof}We shall use Lemma~\ref{lma:elements} and the methods of \cite{SW,BK}. It suffices to treat the case of right regular bands. Let $f:X\to Y$ be a map of right regular bands and let $F=f_!(\star_{PX})$ be the left Kan extension of a terminal presheaf on $(X,\leq)$ along $Uf:(X,\leq)\to(Y,\leq)$. Recall that $F(y)=\pi_0(y{\downarrow} f)$. This defines a factorisation of $Uf$ as $X\to\el_Y(F)\to Y$. By Lemma~\ref{lma:elements}, the poset of elements $\el_Y(F)$ underlies a unique right regular band structure such that $\el_Y(F)\to Y$ is a semigroup homomorphism.

It suffices therefore to show that $X\to\el_Y(F)$ is a semigroup homomorphism as well. The latter takes $x\in X$ to the pair $(f(x),i_x\in F(f(x)))$ where $i_x$ denotes the connected component of $f(x){\downarrow} f$ containing $f(x)\leq f(x)$. We have to show that $(f(x),i_x)(f(x'),i_{x'})=(f(xx'),i_{xx'})$. By definition, $(f(x),i_x)(f(x'),i_{x'})=(f(xx'),i_{x'}\rceil^{f(x')}_{f(xx')})$. Restricting $i_{x'}$ along $f(x')\geq f(xx')$ lands in the connected component $i_{xx'}$.

Connected maps and coverings of right regular bands compose. Moreover, the comprehensive factorisation system for right regular bands is orthogonal because it is so for posets and the forgetful functor $U:\RRB\to\Poset$ is faithful.

For the stability it suffices to show that connected maps are stable under pullback along coverings. This is true for posets,  cf. \cite[Remark 1.11]{BK}, and is thus inherited by right regular bands because the forgetful functor $U:\RRB\to\Poset$ preserves and reflects pullbacks, coverings and connected maps.\end{proof}

%%%%%%%%%%%%%%%%%%%%
\subsection{Normal bands}
%%%%%%%%%%%%%%%%%%%%

A band $X$ is called \emph{normal} if the principal order-ideals ${\downarrow}z=zXz$ are commutative subbands of $X$ for all $z$, i.e. $\forall x,y,z\in X:zxyz=zyxz$.
A band is called left (resp. right) normal if it is simultaneously normal and left (resp. right) regular. According to Kimura \cite{Ki} any normal band is the fibre product of its left and right normal reflections over the semilattice reflection. The following result implies that a normal band is uniquely determined by its semilattice reflection together with two presheaves on it. This observation goes back to Kimura-Yamada \cite{YK} and has recently been recast in modern terms by Kudryavtseva-Lawson \cite{KL}.

\begin{prp}[cf. \cite{YK,KL}]\label{prp:YK}A regular band $X$ is normal if and only if its semilattice reflection $X\to X/\Dd$ is a covering.\end{prp}

\begin{proof}Let $X$ be normal band. Then $X/\Rr$ (resp. $X/\Ll$) is left (resp. right) normal. It suffices now to show that for a left (resp. right) normal band the semilattice reflection is a covering because we then get that $X\to X/\Dd$ is a covering as well, coverings being stable under pullback and under composition. So assume that $X$ is right normal and let $[y]\leq[z]$ in $X/\Dd$. Then, since $[yz]=[y]\wedge[z]=[y]$, the element $yz\leq z$ is a lift in $X$. Moreover, for $y_1,y_2\leq z$ with $[y_1]=[y_2]$, we have $y_1=y_1z=y_2y_1z=y_1y_2z=y_2z=y_2$ whence the covering property of $X\to X/\Dd$.

Conversely, assume that for a regular band $X$ the reflection $X\to X/\Dd$ is a covering. Then $X/\Rr\to X/\Dd$ (resp. $X/\Ll\to X/\Dd$) is a covering of left (resp. right) regular bands. It suffices thus to show that this implies that $X/\Rr$ (resp. $X/\Ll$) is a left (resp. right) normal band. Therefore, without loss of generality, we can assume that $X$ is right regular and $X\to X/\Dd$ is a covering. Pick $x,y,z\in X$. Then $xz\leq z$ and $yz\leq z$. The products $xzyz$ and $yzxz$ have the same image in $X/\Dd$ and hence, by the covering property, they are equal in $X$. But by right regularity, we have $xzyz=xyz$ and $yzxz=yxz$, whence $xyz=yxz$, i.e. $X$ is right normal.\end{proof}

\begin{rmk}\label{rmk:globalpresheaves}A covering of posets $f\colon X\to Y$ is isomorphic to the projection map $\el_Y(F)\to Y$ for a uniquely determined presheaf $F$ on $Y$, defined elementwise by $F(y)=f^{-1}(y)$, cf. Remark~\ref{cover=elementproj}. Accordingly, a right regular band $X$ is normal if and only if there is an isomorphism of right regular bands

\[
X\to \el_{X/\Dd}(F): x\mapsto([x]_\Dd,x)
\]
where the band structure of the poset of elements of $F$ is given by
\[
(a,s)(b,t)=(a\wedge b,t\rceil^b_{a\wedge b}).
\]
The restriction $t\rceil^b_{a\wedge b}$ is uniquely determined by the covering property of $X\to X/\Dd$ and can be made explicit as follows: for $s\in[x]_\Dd$ and $t\in[y]_\Dd$ we get $t\rceil^{[y]_\Dd}_{[x]_\Dd\wedge[y]_\Dd}=st$, cf. the proof of Proposition \ref{prp:YK} and Remark \ref{cover=elementproj}.

%\marginpar{\tiny MG. Maybe we should discuss the equivalence here? CB: we could, but I believe it is more natural below. So, I modified slightly your presentation of Rmk. 1.16. It is crucial that there is no choice for the presheaf structure: it is necessarily restriction, so it is a property, not a structure.}

A similar representation exists for left normal bands. In order to recover a left normal band $X$ from the poset of elements $\el_{X/\Dd}(F)$ of its presheaf, one has however to twist the multiplication $t\rceil^{[y]_\Dd}_{[x]_\Dd\wedge[y]_\Dd}=ts$ for $s\in[x]_\Dd$ and $t\in[y]_\Dd$.

Since any normal band $X$ is a fibre product $X/\Rr\times_{X\Dd}X/\Ll$ of its left and right normal reflections, a normal band $X$ can be represented by a \emph{pair} $(F_1,F_2)$ of presheaves on its semilattice reflection $X/\Dd$ such $X\cong\el_{X/\Dd}(F_1)\times_{X/\Dd}\el_{X/\Dd}(F_2)$ where on one of the two factors multiplication is twisted, cf. \cite[Theorem 5]{YK}.%\marginpar{\tiny CB: is this ok ?}

There is an elegant way to describe the representation theorems for normal, resp. right normal, resp. left normal bands in a uniform way: a regular band $X$ is normal, resp. right normal, resp. left normal if and only if the semilattice reflection $X\to X/\Dd$ is a covering with rectangular, resp. right rectangular, resp. left rectangular fibres. Therefore, $X$ corresponds to the poset of elements of a presheaf taking values in the respective categories of rectangular bands. Now, the categories of left/right rectangular bands are both equivalent to the category of sets with reversed multiplication though, while the category of rectangular bands is equivalent to the category of double sets.\end{rmk}

\begin{cor}\label{YK2}Let $X\to Y$ be a covering of regular bands and assume that $Y$ is normal. Then $X$ is normal as well.\end{cor}

\begin{proof}According to Proposition~\ref{prp:YK}, it suffices to show that $X\to X/\Dd$ is a covering of posets. If we compose with $X/\Dd\to Y/\Dd$ we get a covering. Moreover, since the quotient map $X\to X/\Dd$ is an epimorphism of posets, the induced semilattice homomorphism $X/\Dd\to Y/\Dd$ is a covering. We can now apply the left cancellability of poset coverings to conclude that $X\to X/\Dd$ is a covering as well.\end{proof}

%%%%%%%%%%%%%%%%%%%%
%% The normal reflection of a right regular band
%%%%%%%%%%%%%%%%%%%%

\begin{prp}\label{prp:normalreflection}For any right regular band $X$, the normal reflection $X/\Nn$ may be obtained by comprehensive factorisation of the semilattice reflection $X\to X/\Dd$.\end{prp}

\begin{proof}Let us denote the comprehensive factorisation of the semilattice reflection $X\to X/\Dd$ by $X\to X_N \to X/\Dd$. Then $X_N \to X/\Dd$ is a covering and $X_N$ is a right normal band by Corollary~\ref{YK2}. Therefore, $X_N$ factors through the normal reflection $X/\Nn$ of $X$, where $\Nn$ is the congruence relation on $X$ generated by $(xyz,yxz)$ for $x,y,z\in X$. We get thus the following commutative diagram

$$\xymatrix{X\ar[r]^{conn}\ar[d]&X_N\ar[d]^{cov}\\X/\Nn\ar[r]_{cov}\ar[ru]&X/\Dd}$$in which the upper horizontal map is connected by construction. Again by Proposition~\ref{prp:YK}, the lower horizontal map is a covering as well. The orthogonality of the comprehensive factorisation system implies that the diagonal has a section making the whole diagram commute. Since the quotient map $X\to X/\Nn$ is epimorphic, the diagonal must be invertible.\end{proof}

%%%%%%%%%%%%%%%%%%%%
\section{Distributive bands and distributive skew lattices}\label{sct:skew}
%%%%%%%%%%%%%%%%%%%%

By Proposition~\ref{prp:YK} normal bands are equivalent to presheaves over semilattices.
The aim of the following definition is to capture the subclass of those bands which correspond to sheaves on bounded distributive lattices. We then relate distributive bands to distributive skew lattices in the sense of Leech \cite{Le1,Le2,Le3}.

\subsection{Distributive bands}

Given a band $X$ and a semilattice $L$, we say that $X$ is a band over $L$ provided $X/\Dd\cong L$.
A \emph{zero element} of $X$ is a bottom element of $X$ while a \emph{global element} of $X$ is an element whose reflection is a top element for the meet-semilattice $X/\Dd$.

\begin{dfn}\label{dfn:distband}
A band $X$ is called \emph{distributive} (resp. \emph{Boolean}) provided
\begin{itemize}
\item[(i)]$X$ is a normal band with zero element;
\item[(ii)]each pair of commuting elements of $X$ has a supremum;
\item[(iii)]$X/\Dd$ is a bounded distributive (resp. Boolean) lattice;\end{itemize}
A map of distributive (Boolean) bands is a morphism of the underlying bands preserving zero element, global elements and suprema of commuting elements.\end{dfn}

\begin{rmk}\label{rmk:etale}A regular band $X$ is normal if and only if $X\to X/\Dd$ is a covering, cf. Proposition~\ref{prp:YK}. Condition (ii) may then be interpreted as a lifting of suprema which in our case exist in $X/\Dd$ by condition (iii). Therefore, let us call a covering with this lifting property \emph{etale}. A distributive (Boolean) band may then be defined as an \emph{etale covering} of a distributive (Boolean) lattice. This uniform definition suggests extensions to other classes of lattices, e.g. modular or geometric lattices.

We will use the terminology \emph{right} distributive (resp. \emph{right} Boolean) band if the underlying band is \emph{right regular}. Right Boolean bands have already been studied by Kudryavtseva-Lawson \cite{KL} and shown to form a category equivalent to the category of right-handed \emph{Boolean skew lattices} in the sense of Leech \cite{Le2}. This result does \emph{not} extend to right distributive bands. Conditions under which a right distributive band underlies a right-handed distributive skew lattice are given in Section~\ref{sct:patch} below.\end{rmk}

\begin{rmk}\label{rmk:top}There is an asymmetry regarding ``bottom'' and ``top'' of a distributive band. Any distributive band $X$ has a zero element lying over the bottom element of $X/\Dd$. On the other hand, our assumption that $X/\Dd$ has a top element amounts to the assumption that $X$ has global elements. However, unless $X$ itself is commutative and hence a bounded distributive lattice, there is no top element in $X$. The literature on Boolean bands frequently drops also the assumption that $X/\Dd$ has a top element which forces one to deal with so-called generalised Boolean algebras. We deviate from this practice in order to stay in the realm of classical Stone duality. This simplifies the duality theory but it is not necessary, see e.g. \cite{BCGGK}.\end{rmk}

\begin{lma}\label{lma:sheafppty}
In a distributive band $X$ the following properties hold:
\begin{itemize}
\item[(i)] The quotient map $X\to X/\Dd$ is a morphism of distributive bands.
\item[(ii)]The band operation distributes over binary suprema: for any $x$ and pair $y,z$ of commuting elements one has: $x(y\vee z)=xy\vee xz\text{ and }(y\vee z)x=yx\vee zx.$\item[(iii)]Any finite set of pairwise commuting elements has a supremum.\end{itemize}\end{lma}

\begin{proof} First, note that we can assume that our band $X$ is right normal. For (i),
observe that since $X/\Dd$ is a bounded distributive lattice, it is in particular a distributive band.  Clearly the quotient map preserves zero and global elements. Also, $[x\vee y]_\Dd$ is an upper bound of $[x]_\Dd$ and $[y]_\Dd$. Further, if $x,y\preceq z$, then $x\leq z(x\vee y)$ since $xz(x\vee y)=zx(x\vee y)=x(x\vee y)=x$. Similarly  $y\leq z(x\vee y)$ and thus $x\vee y\leq z(x\vee y)$. It follows that $[x\vee y]_\Dd\leq [z(x\vee y)]_\Dd=[z]_\Dd\wedge[x\vee y]_\Dd$ and thus $[x\vee y]_\Dd\leq [z]_\Dd$. That is, $[x\vee y]_\Dd= [x]_\Dd\vee[y]_\Dd$ as required.

For (ii), by right normality, $yx$ and $zx$ commute so that $yx\vee zx$ exists.  Also, since $y$ and $z$ are less than $y\vee z$ and $X$ is right normal,  $yx$ and $zx$ are less than $(y\vee z)x$, whence $yx\vee zx\leq(y\vee z)x$. Since these elements are $\Dd$-equivalent, it follows that they are equal and thus we obtain the second identity. For the first identity, since $y$ and $z$ commute, it follows that $xy$ and $xz$ commute. Also, since $y\leq y\vee z$, we have $y=y(y\vee z)$ and it follows that $xyx(y\vee z)=xy(y\vee z)=xyy=xy$ so that $xy\leq x(y\vee z)$. Similarily, $xz\leq x(y\vee z)$, whence $xy\vee xz\leq x(y\vee z)$. Again, since these elements are $\Dd$-equivalent, they are in fact equal.

For (iii) note that the supremum of the empty set is the zero element which exists by hypothesis. Assume inductively that $n-1$ pairwise commuting elements have a supremum. It is then enough to show that for $n$ pairwise commuting elements $x_1,\dots,x_n$, the first element $x_1$ commutes with the supremum $x_2\vee\cdots\vee x_n$.

The $n-1$ elements $x_1x_2,x_1x_3,\dots,x_1x_n$ commute pairwise and thus admit a supremum $x_1x_2\vee\cdots\vee x_1x_{n-1}$ which equals $x_1(x_2\vee\cdots\vee x_n)$ by (ii). Similarly, we get $x_2x_1\vee\cdots \vee x_{n-1}x_1=(x_2\vee\cdots\vee x_{n-1})x_1$ from which we obtain the required commutation relation because $x_1$ commutes with $x_2,\dots,x_n$.\end{proof}

\begin{lma}\label{lma:reflection}A morphism of distributive bands $\phi:X\to Y$ induces a morphism of bounded distributive lattices $\phi/\Dd:X/\Dd\to Y/\Dd$.\end{lma}
\begin{proof}By Proposition~\ref{prp:basic}ii, the induced morphism $\phi/\Dd$ is a morphism of meet-semilattices. Since $\phi$ preserves zero and global elements, $\phi/\Dd$ preserves bottom and top elements. We shall now show that $\phi/\Dd$ preserves binary joins.

Let $[x]_\Dd$ and $[x']_\Dd$ be two $\Dd$-classes. Then there is a representative $z\in X$ of the $\Dd$-class $[x]_\Dd\vee[x']_\Dd$ and, by the covering property of $X\to X/\Dd$, there are unique representatives $y,y'$ for $[x]_\Dd,[x']_\Dd$ such that $y\leq z$ and $y'\leq z$.  By normality, we have $yy'=y'y$. Since $X$ is a distributive band, the supremum $y\vee y'$ exists in $X$, and by construction $y\vee y'\leq z$ and $[y\vee y']_\Dd=[z]_\Dd$ so that $y\vee y'=z$. Moreover, since $y$ and $y'$ commute so do $\phi(y)$ and $\phi(y')$ so that, by distributivity of $Y$, the supremum $\phi(y)\vee\phi(y')$ exists and represents the join $[\phi(y)]_\Dd\vee[\phi(y')]_\Dd$ in $Y/\Dd$.

Since $\phi$ preserves suprema of commuting elements $\phi(y\vee y')$ equals $\phi(y)\vee\phi(y')$. Consequently, we have $\phi/\Dd([x]_\Dd\vee[x']_\Dd)=\phi/\Dd([x]_\Dd)\vee\phi/\Dd([x']_\Dd)$.\end{proof}

\subsection{Distributive skew lattices}\label{subsct:skew}

According to Leech \cite{Le1}, a \emph{skew lattice} is a set $S$ with two binary operations $\cm,\cj$ such that $(S,\cm)$ and $(S,\cj)$ are bands related by the following four \emph{absorption laws}
\begin{itemize}\item[(i)]$(y\cm x)\cj x=x=x\cm(x\cj y)$ (i.e. $y\cm x=y$ if and only if $y\cj x=x$);
\item[(ii)]$x\cj(x\cm y)=x=(y\cj x)\cm x$ (i.e. $x\cm y=y$ if and only if $x\cj y=x$).\end{itemize}
In particular, the order relation $y\leq x$ on $(S,\cm)$ is the \emph{opposite} of the order relation on $(S,\cj)$ so that Green's $\Dd$-relation yields on $S/\Dd$ a meet semilattice structure and a join semilattice structure, inducing opposite partial orders on $S/\Dd$. In other words, $S/\Dd$ is a lattice, the \emph{universal lattice reflection} of $S$, cf. \cite{Le1}. Both bands $(S,\cm)$ and $(S,\cj)$ are regular, and $(S,\cm)$ is left (resp. right) regular if and only if $(S,\cj)$ is right (resp. left) regular in which case the skew lattice $S$ is said to be \emph{left-handed} (resp. \emph{right-handed}). Every skew lattice is the fibre product of its left- and right-handed reflections over its lattice reflection, cf. \cite{Le1}.

A skew lattice $(S,\cm,\cj)$ is called \emph{normal} if $(S,\cm)$ is a normal band. This amounts to the property that the principal downsets ${\downarrow} z$ are \emph{lattices} for all $z\in S$. A skew lattice $S$ is called \emph{symmetric} if $x\cm y=y\cm x$ precisely when $x\cj y=y\cj x$.

We call a skew lattice $S$ \emph{distributive} (resp. \emph{Boolean}) if it is a \emph{normal} and \emph{symmetric} skew lattice whose lattice reflection $S/\Dd$ is \emph{bounded distributive} (resp. \emph{Boolean}). Both classes of skew lattices form \emph{varieties}, cf. \cite{Le1,Le2}.

Leech \cite{Le1} uses a weaker notion of distributivity for which neither symmetry nor normality is required, yet coming equipped with a distributive lattice reflection. We refer the interested reader to \cite{KLP} for a comparison of several notions of distributivity. Our notion is called \emph{strongly distributive} or \emph{meet bidistributive} in \cite{BCGGK, CKLS, KLP, Le3} and it is essential to our results. Since we only use this strong form of distributivity, we suppress additional adjectives, hoping no confusion will arise. Note also that we assume the lattice reflection to be \emph{bounded}. This is not standard and could be avoided, but it simplifies the duality theoretic aspects of our work.

Below, we give a new proof of Leech's equational characterisation of distributive skew lattices among normal skew lattices, cf. Proposition~\ref{prp:distributive}. On the way, we characterise symmetric skew lattices among normal skew lattices, cf. Lemma~\ref{lma:symmetric}.

\begin{lma}\label{lma:normal}A normal skew lattice $S$ has a distributive lattice reflection $S/\Dd$ if and only if for all $x,y,z\in S$,$$z\cm(x\cj y)\cm z=(z\cm x\cm z)\cj(z\cm y\cm z).$$

A normal left-handed (resp. right-handed) skew lattice has a distributive lattice reflection if and only if for all $x,y,z\in S$,$$z\cm(x\cj y)=(z\cm x)\cj(z\cm y)\quad (\text{resp.}\quad (x\cj y)\cm z=(x\cm z)\cj(y\cm z))$$\end{lma}
\begin{proof}The three elements $z\cm (x\cj y)\cm z$, $z\cm x\cm z$ and $z\cm y\cm z$ belong to the downset ${\downarrow} z=z\cm S\cm z$. For a normal skew lattice this downset is a lattice upon which the reflection $S\to S/\Dd$ is injective. This establishes the identity if $S/\Dd$ is distributive. Conversely, the identity clearly implies that $S/\Dd$ is distributive.

The left- (resp. right-handed) case follows immediately from the general case.\end{proof}

\begin{exm}[Non-symmetric skew lattices and free regular bands]\label{exm:freerightregular}--\vspace{1ex}

For a finite set $A$, let $S(A)\subset A^*$ be the set of those words of the free monoid $A^*$ on $A$ containing each element of $A$ at most once. This inclusion admits a retraction $\rho:A^*\to S(A)$ (in sets) assigning to a word $w\in A^*$ the subword $\rho(w)\in S(A)$ obtained from $w$ by keeping the \emph{rightmost} occurrence of each element of $A$ in $w$. This induces a \emph{non-commutative join} operation on $S (A)$:$$s\cj_r t=\rho(st)$$turning $(S(A),\cj_r)$ into a right regular band. In fact, this is the free unital right regular band generated by $A$, cf. \cite[Theorem 10]{Ki}.

We can think of $S(A)$ as the set of \emph{linearly ordered subsets} of $A$. For two such linearly ordered subsets $B,C$ of $A$, the join  $B\cj_r C$ is then the subset $B\cup C$ equipped with the ordinal sum of the induced orders on $B\setminus C$ and $C$ respectively.%\marginpar{\tiny  MG: On S(A) the order is reverse \emph{suffix} right? CB: yes, MG: I added stuff about linear orders. Ok? CB: yes}

The semilattice quotient $S(A)/\Dd$ is isomorphic to the powerset $\Pp(A)$ of $A$, equipped with union as binary operation. Note that the partial order on $(\Pp(A),\cup)$ is the reverse of subset-inclusion, and the partial order on $(S(A),\cj_r)$ is the reverse of ``suffix'' relation, i.e. the reverse of right interval inclusion.

There is a ``dual'' left regular operation $\cj_l$ on $S(A)$, obtained by means of the retraction $\lambda:A^*\to S(A)$ keeping the \emph{leftmost} occurrence of each element of $A$. This yields $(S(A),\cj_l)$, the free unital left regular band on $A$, cf. \cite[Theorem 10]{Ki}.

There is also a \emph{non-commutative meet} operation of $S(A)$, defined by letting $s\cm_r t$ be the subword of $t$
 consisting of those elements of $A$ which occur in $s$ as well. Viewed as linearly ordered subsets, $B\cm_r C$ is the intersection $B\cap C$ ordered according to $C$. This makes $(S(A),\cm_r)$ into a right regular band, where this time the semilattice quotient $S(A)/\Dd$ is isomorphic to $(\Pp(A),\cap)$ and the partial orders are subword inclusion, resp. monotone subset inclusion.

The quotient map $(S(A),\cm_r)\to(\Pp(A),\cap)$ is easily seen to be a \emph{covering} so that $(S(A),\cm_r)$ is a right normal band by Proposition~\ref{prp:YK}. This also follows directly from the definitions. The fibre over a subset $B\subseteq A$ of this quotient map corresponds to the set of all possible linear orderings of $B$.

It is then a straightforward verification that $S^r(A)=(S(A),\cm_r,\cj_l)$ is a \emph{right-handed} normal skew lattice whose lattice reflection is $(\Pp(A),\cap,\cup)$. This skew lattice is however \emph{not} symmetric. Elements that commute for $\cm_r$ do not necessarily commute for $\cj_l$, e.g.,  any two elements of $S(A)$ with disjoint images in $\Pp(A)$ commute for $\cm_r$ but do not so for $\cj_l$ unless one of them is the empty word.

A similar argument shows that the ``dual'' construction $S^l(A)=(S(A),\cm_l,\cj_r)$ yields a \emph{left-handed} normal skew lattice which is not symmetric.

The fibre product $\tilde{S}(A)=S^l(A)\times_{\Pp(A)} S^r(A)$ yields then a normal non-symmetric skew lattice. Note that its underlying band $(\tilde{S}(A),\cj)$ is the free unital regular band generated by $A$, cf. \cite[Theorem 11]{Ki}. \end{exm}

Despite the interesting example just treated, for our intended sheaf-theoretical interpretation, the symmetry of a skew lattice will play a crucial role. The following lemma reformulates symmetry of a normal skew lattice in equational terms. In \cite[Theorem 4.2]{CKLS} a related (but more complicated) equational description of general symmetric skew lattices is given.

\begin{lma}\label{lma:symmetric}
A normal skew lattice $S$ is symmetric if and only if for all $x,y\in S$,
\[
(x\cm y)\cj x=x\cm(y\cj x)\text{ and }(x\cj y)\cm x=x\cj(y\cm x).
\]
\end{lma}

\begin{proof}
Note that since $(S,\cj)$ is regular, if $x\cj y=y\cj x$ then, by Lemma~\ref{lma:commuting}, the join $x\vee y$ exists, and $x,y$ belong to the downset ${\downarrow}(x\vee y)$ so that normality of $(S,\cm)$ implies $x\cm y=y\cm x$. That is, in any normal skew lattice we have that $x\cj y=y\cj x$ implies $x\cm y=y\cm x$. It is thus sufficient to deal with the other implication of symmetry.

Also, it suffices to prove the statement separately for left- and right-handed skew lattices. Since the proofs are symmetric, we will assume our skew lattice is right-handed. In this case, the right identity holds because both sides equal $x$: By left regularity of $(S,\cj)$, $x\leq x\cj y$ and thus by right regularity of $(S,\cm)$, $x=(x\cj y)\cm x$; symmetrically, $x=x\cj(y\cm x)$.

Thus we must show that for a right-handed normal skew lattice, the left identity is equivalent to $x\cm y=y\cm x$ implies $x\cj y=y\cj x$ for all $x,y\in S$. To this end, assume first that the left identity holds. Then, if $x\cm y=y\cm x$, we get by absorption $x=(y\cm x)\cj x=(x\cm y)\cj x=x\cm(y\cj x)$ so that $x\leq y\cj x$. Since for a right-handed skew lattice $S$, the band $(S,\cj)$ is left regular, we also have $y\leq y\cj x$. By normality the downset ${\downarrow}(y\cj x)$ is a lattice so that $y\cj x$ represents the supremum $y\vee x$. Exchanging the roles of $x$ and $y$ we get $x\cj y=x\vee y=y\vee x=y\cj x$, i.e. symmetry.

Assume conversely that $S$ is symmetric. Since $S$ is a right handed skew lattice, we have $x\cm y\leq(x\cm y)\cj x$ and $x\cm y\leq y$ and $[x\cm y]_\Dd$=$[(x\cm y)\cj x]_\Dd\wedge[y]_\Dd$. Therefore, Lemma~\ref{lma:commuting} implies that $y$ and $(x\cm y)\cj x$ $\cm$-commute so that by symmetry they also $\cj$-commute, yielding
\begin{align*}
((x\cm y)\cj x)\cj(y\cj x)&=[((x\cm y)\cj x)\cj y]\cj x\\
				  &=[y\cj ((x\cm y)\cj x)]\cj x\\
				  &=[y\cj (x\cm y)]\cj x=y\cj x.
\end{align*}
Whence $(x\cm y)\cj x\leq y\cj x$. Since $(S,\cm)$ is right regular, we also have $x\cm(y\cj x)\leq y\cj x$ so that the left identity follows from normality.\end{proof}

\begin{prp}[cf. \cite{Le3}, Theorem 2.5]\label{prp:distributive}A normal skew lattice $S$ is symmetric and has a distributive lattice reflection if and only if for all $x,y,z\in S$,$$x\cm(y\cj z)=(x\cm y)\cj(x\cm z)\text{ and }(x\cj y)\cm z=(x\cm z)\cj(y\cm z).$$\end{prp}
\begin{proof}Again, left- and right-handed skew lattices can be treated separately, and the proofs are symmetric to each other. We give the proof for right-handed skew lattices. Assume first that the two distributivity laws hold. Then by Lemma~\ref{lma:normal} the right law yields distributivity of the lattice reflection, and by Lemma~\ref{lma:symmetric} the left law with $z=x$ yields symmetry.

Conversely, if the skew lattice is symmetric and has distributive lattice reflection, then Lemma~\ref{lma:normal} yields the right law while Lemma~\ref{lma:symmetric} implies$$(x\cm y)\cj(x\cm z)\leq(x\cm y)\cj z=x\cm(y\cj z)\leq y\cj z$$and the left law follows from normality, the two sides being $\Dd$-equivalent.\end{proof}

\begin{cor}\label{cor:forgetful}Every distributive (resp. Boolean) skew lattice $(S,\cm,\cj)$ has an underlying distributive (resp. Boolean) band $(S,\cm)$.\end{cor}
\begin{proof}It is enough to check condition~\ref{dfn:distband}ii of a distributive $\cm$-band. For any two $\cm$-commuting elements, symmetry of the skew lattice implies $x\cj y=y\cj x$ whence, by Lemma~\ref{lma:commuting}, the latter element represents the supremum $x\vee y$.\end{proof}

For ease of terminology we shall say right Boolean/distributive skew lattice for a right-handed Boolean/distributive skew lattice.
 In view of Lemma~\ref{lma:normal}, right Boolean skew lattices are the same as \emph{skew Boolean algebras} in the terminology of Leech \cite{Le2}. The following proposition shows that a right Boolean skew lattice is completely determined by its underlying right Boolean band. This has been shown by Kurdryatseva-Lawson \cite{KL} in a slightly more general context (not requiring the existence of global elements).

\begin{prp}[cf. \cite{KL}, Theorem 1.11]\label{prp:Boolean}The forgetful functor from right Boolean skew lattices to right Boolean bands is an equivalence of categories.\end{prp}

\begin{proof} Note first that the right Boolean skew lattice $(S,\cm,\cj)$ is completely determined by its underlying right Boolean band $(S,\cm)$. Indeed, for $x,y\in S$ we have the identity $x\cj y=x\vee(y\backslash x)$ where $y\backslash x$ denotes the relative complement of $x$ in $y$, cf. \cite[Lemma 2.4]{KL}. Since $x$ and $y\backslash x$ $\cm$-commute, their supremum is already defined in the right Boolean band $(S,\cm)$ so that the binary join operation $\cj$ is completely determined by the underlying $\cm$-band structure.

Let us construct $(S,\cm,\cj)$ out of $(S,\cm)$. For $x,y\in S$, we have $x\cm y\leq y$, i.e. $x\cm y\in{\downarrow}y$. Therefore, $[x\cm y]_\Dd$ has a relative complement $[y]_\Dd\backslash[x]_\Dd\cm[y]_\Dd$ which lifts uniquely to $S$ and thus defines an element $y\backslash x\in S$ with the formal properties of a relative complement
%\footnote{Our notation for relative complements is ``set-theoretical'' and might be confusing for readers familiar with residuated lattices where the same notation has a different meaning.}. MG I was the one who wanted this but I think it is ok without it afterall...
In particular, $x$ and $y\backslash x$ $\cm$-commute so that they have a supremum in $S$, which by definition is our join $x\cj y$. We now have to check that the four absorption laws of a skew lattice hold and that the latter is symmetric. For the symmetry it is enough to establish that $x\cm y=y\cm x$ implies $x\cj y=y\cj x$. This implication holds because our construction of the binary join yields $x\cj y=x\vee y=y\cj x$ whenever $x$ and $y$ $\cm$-commute.

For the absorption laws we have to show that $y\cm x=y$ if and only if $y\cj x=x$ as well as $x\cm y=y$ if and only if $x\cj y=x$. The first equivalence holds because both sides express $y\leq x$. For the second equivalence, if $x=x\cj y$, i.e. $x=x\vee y\backslash (x\cm y)$, then by Lemma~\ref{lma:sheafppty}i we have $x\cm y=(x\vee y\backslash x\cm y)\cm y=(x\cm y)\vee(y\backslash x\cm y)=y$. Conversely, if $x\cm y=y$ then $x\cj y= x\vee y\backslash(x\cm y)=x\vee (y\backslash y)=x$.

Finally, Boolean band morphisms preserve the constructed binary joins as they preserve joins of commuting elements and their semilattice reflections preserve the bounded lattice operations and thus also the relative complements.\end{proof}

One may wonder whether Proposition~\ref{prp:Boolean} holds for right distributive skew lattices as well. This is not the case. The following lemma shows how to construct counter-examples. There are non-commutative right distributive bands with a single global element over any finite linear order with more than two elements (viewed as a distributive lattice).

\begin{lma}Any right distributive skew lattice with a unique global element is commutative.\end{lma}

\begin{proof}Let $a$ be the unique global element and $b,c$ be two $\Dd$-equivalent elements, i.e. $b\cm c=c$ and $c\cm b=b$. Since $a$ is the unique global element and the lattice reflection preserves binary joins, we have $b\cj a=a$ and $c\cj a =a$. Consequently, by absorption, we have $b=b\cm(b\cj a)=b\cm a$ and similarily, $c=c\cm a$. Therefore, by right normality, $c=b\cm c=b\cm c\cm a=c\cm b\cm a=c\cm b=b$, i.e. all $\Dd$-classes are singleton and the skew lattice is commutative.\end{proof}

%%%%%%%%%%%%%%%%%%%%
\section{Stone duality for spectral sheaves}\label{sct:duality}
%%%%%%%%%%%%%%%%%%%%
This section contains our first main result: we establish a duality between global sheaves on spectral spaces and right distributive bands. Let us first review the relevant definitions concerning sheaves.

\subsection{Sheaves}A (set-valued) \emph{presheaf} on a topological space $\Xx$ is a functor

$$\Ff\colon\Omega(\Xx)^\op\to \Set$$ where the set $\Omega(\Xx)$ of open subsets of $\Xx$ is viewed as a posetal category. The functoriality of the presheaf yields restriction maps $|^U_V\colon \Ff(U)\to \Ff(V)$ for $V\subseteq U$ which compose in the evident way. Note that, in order to lighten the notation, we will write $|_V$ for $|^U_V$ whenever the domain, in this case $\Ff(U)$, can easily be inferred from the context.

For an open cover $\Uu$ of an open subset $W$ of $\Xx$, a family $s_U\in \Ff(U)$ indexed by $U\in\Uu$ is called \emph{compatible} provided $s_U|_{U\cap V}=s_V|_{U\cap V}$ for all $U,V$ in $\Uu$.

A presheaf $F$ on $\Xx$ is called a \emph{sheaf} provided, for any such compatible family $\{s_U\}_{U\in\Uu}$ on any open subset $W$, there is a unique $s\in \Ff(W)$ with $s|_U=s_U$ for all $U$ in $\Uu$.
The notion of sheaf on a topological space only uses the \emph{frame} structure of $\Omega(\Xx)$, i.e. that $\Omega(\Xx)$ has finite meets and arbitrary joins which distribute over finite meets. In particular, it makes sense to define sheaves on any frame.

We may also define a notion of sheaf over a bounded distributive lattices $L$ but in this setting we will only require the gluing condition for \emph{finite} compatible families of sections of $F$. To be precise: a presheaf $\Ff\colon L^\op\to \Set$ is called a sheaf if for any \emph{finite} set $A$ of elements of $L$, and any family $s_a\in \Ff(a)$ indexed by $a\in A$ such that $s_a|_{a\wedge b}=s_b|_{a\wedge b}$ for  $a,b\in S$, there is a unique $s\in \Ff(\bigvee A)$ such that  $s|_a=s_a$  for $a\in A$. We denote by $\Sh(L)$ the category of set-valued sheaves on $L$. It is an example of a \emph{coherent} topos, cf. Mac Lane-Moerdijk \cite{MoerMacL}.
%\marginpar{\tiny MG: when referring to a book, I think one should be more precise in the reference. CB: yes of course: I do not have access to tnhe book.}

Recall that a (pre)sheaf $\Ff:L^\op\to\Sets$ is said to be \emph{global} provided there is a map from the terminal (pre)sheaf to $\Ff$, cf. Remark~\ref{rmk:globalpresheaves}. When $L$ has a top element, $\Ff$ is global if and only if the set $\Ff(1)$ of global sections is non-empty. In the following proposition we have to restrict to global (pre)sheaves because the semilattice reflection of a distributive band is bounded by definition, cf. Remark~\ref{rmk:top}.

\begin{prp}\label{prp:localsectionband}For a presheaf $\,\Ff$ on a bounded distributive (resp. Boolean) lattice $L$, the category of elements $\el_L(\Ff)$ is a right distributive (resp. Boolean) band over $L$ if and only if $~\Ff$ is a global sheaf. \end{prp}
%\marginpar{\tiny MG: Could even say: For a presheaf $F$ on a bounded distributive lattice $L$, $\el_L(F)$ is right distributive band over $L$  iff $F$ is a global sheaf. Right? CB: yes.}

\begin{proof}Let $\Ff$ be a presheaf on a bounded distributive lattice $L$. An object $(U,s)$ of $\el_L(\Ff)$ is by definition a local section $s\in \Ff(U)$. By Proposition~\ref{prp:YK}, there is a uniquely defined multiplication $$(U,s)(V,t)=(U\wedge V,t|_{U\wedge V})$$ turning $\el_L(\Ff)$ into a \emph{right normal band} whose semilattice reflection $\el_L(\Ff)/\Dd$ embeds into $L$ as a downset, cf. Lemma~\ref{lma:elements}. Therefore, $\Ff$ is a global presheaf if and only if the quotient $\el_L(\Ff)/\Dd$ contains a top element if and only if it equals $L$.

The partial order on the right normal band $\el_L(\Ff)$ corresponds to restriction of local sections: $(V,t)\leq(U,s)$ just means that $V\leq U$ and $t=s|_{V}$. Therefore, local sections $(U,s),(V,t)$ commute in $\el_L(\Ff)$ if and only if their restrictions to the meet $U\wedge V$ agree, cf. Lemma~\ref{lma:commuting}. The presheaf $\Ff$ satisfies thus the sheaf condition for pairs of elements of $L$ if and only if commuting local sections $(U,s),(V,t)$ have a supremum $(U,s)\vee(V,t)$ in $\el_L(\Ff)$, defined on the join $U\vee V$ by gluing. This last condition amounts by Lemma~\ref{lma:sheafppty} to the sheaf condition for \emph{finite} compatible families of local sections. Therefore, the category of elements satisfies Condition \ref{dfn:distband}ii of a right distributive (resp. Boolean) band if and only if $\Ff$ is a global sheaf.\end{proof}

\begin{cor}\label{cor:sheafification}Let $L$ be a bounded distributive (resp. Boolean) lattice. The category of right distributive (resp. Boolean) bands over $L$ is a reflective subcategory of the category of right normal bands over $L$. The reflection preserves finite meets.\end{cor}

\begin{proof}Right normal bands over $L$ form a category equivalent to global presheaves on $L$, cf. Remark \ref{rmk:globalpresheaves}. By Proposition \ref{prp:localsectionband}, this equivalence restricts to one between right distributive (resp. Boolean) bands over $L$ and global sheaves on $L$. The claim now follows from the fact that sheafification is a finite limit preserving reflection.\end{proof}

\begin{rmk}\label{rmk:reflective}The following proposition generalises the existence of a Boolean envelope for bounded distributive lattices to right distributive bands.%\marginpar{\tiny CB: I put the proposition here where it makes more sense than in a separate section.}

Recall that the \emph{Boolean envelope} $L_B$ of a bounded distributive lattice $L$ comes equipped with a lattice homomorphism $L\to L_B$ enjoying the following universal property: each homomorphism of bounded distributive lattices $L\to B$ whose codomain $B$ is a Boolean lattice factors uniquely through $L\to L_B$. In more abstract terms, the assignment $L\mapsto L_B$ defines a functor from the category of bounded distributive lattices to the category of Boolean lattices which is \emph{left adjoint} to the inclusion functor. Such a left adjoint functor is called a \emph{reflection}, and the category of Boolean lattices is said to be a \emph{reflective} subcategory of the category of bounded distributive lattices. The faithfulness of the reflection functor is equivalent to the property that the homomorphisms $L\to L_B$ are monic for all $L$.\end{rmk}

\begin{prp}\label{prp:Bool1}The category of right Boolean bands is a reflective subcategory of the category of right distributive bands.\end{prp}

\begin{proof}Let $S$ be a right distributive band with bounded distributive lattice reflection $S\to S/\Dd$. We compose with the Boolean envelope $S/\Dd\to(S/\Dd)_B$ and factor the composite as a connected morphism followed by covering (cf. Proposition~\ref{prp:factorisation})$$\xymatrix{S\ar[r]^{connected}\ar[d]&S'\ar[d]^{covering}\ar[r] &S'_B\ar[d]^{etale\, covering}\\S/\Dd\ar[r]&(S/\Dd)_B\ar@{=}[r]&S'_B/\Dd}$$and finally apply sheafification to the presheaf represented by the middle vertical map in order to get an etale covering on the right by Corollary~\ref{cor:sheafification}, cf. Remark \ref{rmk:etale}.

To check the universal property of the composite reflection $S\to S'_B$, let $S\to T$ be a map of distributive bands whose codomain $T$ is a right Boolean band. By Lemma~\ref{lma:reflection} we get a homomorphism of distributive lattices $S/\Dd\to T/\Dd$ whose codomain is a Boolean lattice. Therefore, thanks to the universal property of the Boolean envelope, we get a unique factorisation $S/\Dd\to(S/\Dd)_B\to T/\Dd$. The latter induces a compatible map of right normal bands $S'\to T$ because of the orthogonality between connected maps and coverings. Since $T\to T/\Dd$ is an etale covering, the given map $S\to T$ factors uniquely through $S'_B$ by Corollary~\ref{cor:sheafification}.\end{proof}

%%%%%%%%%%%%%%%%%%%%
\subsection{Stone duality}\label{subsct:stone}
%%%%%%%%%%%%%%%%%%%%

Marshall Stone \cite{Stone37} has shown that the category of bounded distributive lattices is dually equivalent to the category of spectral spaces.

A \emph{spectral space} is a sober topological space $\Xx$ such that the collection $\KO(\Xx)$ of \emph{compact open} subsets forms a bounded distributive sublattice  of $\Omega(\Xx)$ generating the topology of $\Xx$. Spectral spaces are sometimes called ``Stone spaces'' (see e.g. Balbes-Dwinger \cite{BalbesDwinger} or Nerode \cite{Nerode}), but nowadays the name Stone space is mainly used for Boolean spaces, cf. Stone \cite{Stone36} and Johnstone \cite{Johnstone}. Hochster \cite{Hoch} has shown that spectral spaces are precisely those spaces which occur as prime spectra of commutative rings, and this may serve as justification for our terminology. The appropriate maps between bounded distributive lattices are lattice homomorphisms preserving the bounds, and between spectral spaces they are the continuous maps for which the preimage of a compact open subset is compact open. The latter are called \emph{spectral maps}. For a recent reference on spectral spaces see \cite{DST}.

For a bounded distributive lattice $L$, the dual spectral space $\Spec(L)$ is the set of bounded lattice homomorphisms $L\to\{0,1\}$ equipped with the subspace topology of $\{0,1\}^L$ where $\{0,1\}$ carries the Sierpinski topology.
%\[\widehat{a}=\{\pp\in\Spec(L)\mid a\in \pp\}\]for $a\in L$. On the other hand,
For a spectral space $\Xx$, the dual lattice is the set $\KO(\Xx)$ of compact open subsets of $\Xx$ viewed as a bounded distributive sublattice of $\Omega(\Xx)$. For more details see \cite{BalbesDwinger, Johnstone}.

Recall that a topological space $\Xx$ is called \emph{sober} if every irreducible closed subset of $\Xx$ is the closure of a unique point. This is weaker than Hausdorff's separation axiom $T_2$ but stronger than Kolmogorov's separation axiom $T_0$. The importance of sobriety comes from the fact that, assuming \emph{Zorn's Lemma}, a topological space $\Xx$ is sober if and only if the map $\Xx\to\Hom_\Frm(\Omega(\Xx),\{0,1\})$ taking a point of $\Xx$ to its characteristic frame homomorphism is a bijection. In particular, assuming Zorn's Lemma, the category of sober spaces and continuous maps is dually equivalent to the category of spatial frames and frame homomorphisms. Combined with the identification of bounded distributive lattices and coherent frames, this yields a possible proof of the aforementioned Stone duality, cf. Johnstone \cite{Johnstone}.%\marginpar{\tiny CB: Is this ok ?}

A nice application of Stone duality is a topological perspective on the Boolean envelope $L_B$ of a bounded distributive lattice $L$, cf. Remark \ref{rmk:reflective} and Nerode \cite[Theorem 2.1]{Nerode}. Indeed, by Stone duality, $L$ may be represented as the dual lattice $\KO(\Xx)$ of a spectral space $\Xx$. The Boolean envelope $\KO(\Xx)_B$ may then be identified with the dual lattice $\KO(\Xx^p)$ of a Boolean space $\Xx^p$ canonically associated with the spectral space $\Xx$, namely the one obtained by replacing the \emph{spectral} topology with the finer \emph{patch} topology, cf. Section \ref{subsct:patch} below. The latter is generated by the so-called \emph{constructible} subsets of $\Xx$. In other words, the Boolean envelope $\KO(\Xx)_B=\KO(\Xx^p)$ coincides with the Boolean subalgebra of the powerset $\Pp(\Xx)$ generated by the compact open subsets of $\Xx$.%\marginpar{\tiny CB: Is this ok ?}

\subsection{Spectral sheaves}\label{subsct:spectralsheaves}We shall call \emph{spectral sheaf} any sheaf on a spectral space and \emph{map of spectral sheaves} any pair $(f,\phi):(\Xx,\Ff)\to(\Xx',\Ff')$, where $f:\Xx\to\Xx'$ is a spectral map between spectral spaces and $\phi_U:\Ff'(U)\to \Ff(f^{-1}(U))$ is a family of maps indexed by $U\in\Omega(\Xx')$ which is natural with respect to restrictions. Such a family determines, and is determined, either by a map of spectral sheaves $\Ff'\to f_*\Ff$ over $\Xx'$, or by a map of spectral sheaves $f^*\Ff'\to \Ff$ over $\Xx$.
%\marginpar{\tiny MG: I assume this relies on a generally known fact. CB: I do not have access to MacLane-Moerdijk. It is just the adjoint pair $(f^*,f_*)$ which must be treated at the very beginning in loc. cit.}
Spectral sheaves and maps of spectral sheaves form a category containing the category of spectral spaces and spectral maps as a full subcategory.

Any sheaf $\Ff$ on a spectral space $\Xx$ can be restricted to the dual lattice $\KO(\Xx)$ consisting of the compact opens of $\Xx$. Since the latter is a basis of the topology, this restriction induces an equivalence of categories between the respective sheaf categories, cf. \cite[Theorem~II.1.3]{MoerMacL}. Since in $\Omega(\Xx)$ every cover has a finite subcover, we get the following proposition which will be central to our further work.

\begin{prp}\label{prp:sheavesDL}Let $\Xx$ be a spectral space and $\,L=\KO(\Xx)$ its dual lattice. Then the category of sheaves on $\Xx$ is equivalent to the category of sheaves on $L$.\end{prp}

The following is our first main result. It extends classical Stone duality \cite{Stone37}.

\begin{thm}\label{thm:main0}The category of spectral sheaves with global sections is dually equivalent to the category of right distributive bands.\end{thm}

\begin{proof}
Proposition~\ref{prp:YK} yields a correspondence between right normal bands $X$ and global presheaves on meet-semilattices. By Proposition~\ref{prp:localsectionband} this correspondence restricts to one between right distributive bands $X$ and global sheaves over bounded distributive lattices $X/\Dd$. By Proposition~\ref{prp:sheavesDL} sheaves on $L=X/\Dd$ correspond to sheaves on $\Xx=\Spec(X/\Dd)$, i.e. to spectral sheaves. This establishes the asserted equivalence on the level of objects.

Every map of right distributive bands $\phi:X\to X'$ induces by Lemma~\ref{lma:reflection} a commutative diagram$$\xymatrix{X\ar[r]^\phi\ar[d]&X'\ar[d]\\X/\Dd\ar[r]_{\phi/\Dd}&X'/\Dd}$$in which the lower horizontal map is a morphism of bounded distributive lattices. Stone duality yields then a spectral map $f:\Spec(X'/\Dd)\to\Spec(X/\Dd)$. Note that we have identifications $\KO(\Spec(X'/\Dd))=X'/\Dd$ and $\KO(\Spec(X/\Dd))=X/\Dd$ under which $f^{-1}$ identifies with $\phi/\Dd$. Consequently, the upper horizontal map $\phi:X\to X'$, once identified with $\el_{X/\Dd}(\Ff)\to\el_{X'/\Dd}(\Ff')$, yields maps of local sections $\rho_U:\Ff(U)\to \Ff'(f^{-1}(U))$ which are natural with respect to restrictions. We have thus constructed a contravariant functor from the category of right distributive bands to the category of global spectral sheaves.

Every map $(f,\rho):(\Xx',\Ff')\to(\Xx,\Ff)$ of global spectral sheaves induces by Stone duality a map of bounded distributive lattices $\KO(\Xx)\to\KO(\Xx')$ as well as distributive bands $\el_{\KO(\Xx)}(\Ff)$ and $\el_{\KO(\Xx')}(\Ff')$. The natural transformation $\rho_U:\Ff(U)\to \Ff'(f^{-1}(U))$ restricts to a map of posets $\el_{\KO(\Xx)}(\Ff)\to\el_{\KO(\Xx')}(\Ff')$ compatible with the lattice reflections. The uniqueness of the band multiplication (cf. Proposition~\ref{prp:YK}) implies that this map of posets is a map of right normal bands. The compatibility with the lattice reflections implies moreover that zero and global elements are preserved, as well as suprema of commuting elements, i.e. we have constructed a contravariant functor from global spectral sheaves to right distributive bands. The two functors are mutually quasi-inverse.\end{proof}

\begin{rmk}The duality of Theorem~\ref{thm:main0} comprises in particular an equivalence between the category of global sheaves on a fixed spectral space $\Xx$ and the category of right distributive bands over the dual lattice $\KO(\Xx)$. The contravariant aspect of Theorem~\ref{thm:main0} concerns base change and reduces here to classical Stone duality.

In the special case where $\Xx$ is a Boolean space, i.e. the dual lattice $\KO(\Xx)$ is a Boolean algebra, other representations of global sheaves exist. Our choice to take the right Boolean band of local sections of $\Ff$ is shared by Kudryavtseva-Lawson \cite{KL}. For related representations see Bergman \cite{Berg} and Burris-Werner \cite{BW}.
\end{rmk}
In Section \ref{sct:patch} we will characterise in sheaf-theoretical terms those global sheaves on spectral spaces whose band of local sections extends to a right distributive skew lattice. The following proposition is a tiny step in this direction.

\begin{prp}\label{prp:skewsheaf}Let $\Xx$ be the spectral space and $\Ff:\KO(\Xx)^\op\to\Set$ be a sheaf whose band of local sections extends to a right distributive skew lattice $(S_\Ff,\cm,\cj)$.

\begin{itemize}\item[(i)]For compact open subsets $U,V,W\in\KO(\Xx)$ such that $U\supseteq V$, the square
$$\xymatrix{\Ff(V)\times \Ff(U)\ar[r]^\cj\ar[d]_{|^V_{V\cap W}\times|^U_{U\cap W}}&\Ff(U)\ar[d]^{|^U_{U\cap W}}\\\Ff(V\cap W)\times \Ff(U\cap W)\ar[r]_>>>>>\cj& \Ff(U\cap W)}$$is commutative.
\item[(ii)]The sheaf $\Ff$ is quasi-flasque (cf. Kempf \cite{Kempf}), i.e. for all compact open subsets $V\subseteq U$, the restriction map $|^U_V:\Ff(U)\to \Ff(V)$ is surjective.
    \item[(iii)]For all compact open subsets $V\subseteq U$, the restriction map $|^U_V:\Ff(U)\to \Ff(V)$ may be identified with a projection map $\Ff(V)\times\Ff_{U\backslash V}(U)\to\Ff(V)$ whose fibre is definable by the following coequaliser diagram$$\xymatrix{\Ff(V)\times\Ff(U)\ar@<-.5ex>[r]_>>>>{\cj} \ar@<.5ex>[r]^>>>>{pr_{2}} & \Ff(U)\ar[r]&\Ff_{U\backslash V}(U)}$$ in the category of sets.\end{itemize}\end{prp}

\begin{proof}(i) Since the binary join of $S_\Ff$ is preserved by $S_\Ff\to S_\Ff/\Dd$, the horizontal maps of the square are well-defined. Now take any $(x,y,z)\in \Ff(U)\times \Ff(V)\times \Ff(W)$. By Remark \ref{rmk:globalpresheaves}, commutativity of the square amounts to the identity $z\cm(y\cj x)=(z\cm y)\cj(z\cm x)$ in $S_\Ff$ which holds by Proposition~\ref{prp:distributive}.

(ii) Taking $V=W$ and $y=z$ in (i) shows that $(y\cj x)|^U_V=y$.

(iii) For $(y,x)\in\Ff(V)\times\Ff(U)$ the absorption law $(x\cm y)\cj y=y$ implies that the parallel pair above has a common section $(|^U_V,id_{\Ff(U)}):\Ff(U)\to\Ff(V)\times\Ff(U)$, i.e. the coequaliser is \emph{reflexive}. Reflexive coequalisers in the category of sets are preserved under Cartesian product. Applying $\Ff(V)\times(-)$ we thus get$$\xymatrix{\Ff(V)\times\Ff(V)\times\Ff(U)\ar@<-.5ex>[r]_>>>>{id_{\Ff(V)}\times\cj} \ar@<.5ex>[r]^>>>>{pr_{13}} & \Ff(V)\times\Ff(U)\ar[r]&\Ff(V)\times\Ff_{U\backslash V}(U)}$$ a coequaliser diagram in sets. On the other hand, we also have a \emph{split} coequaliser$$\xymatrix{\Ff(V)\times\Ff(V)\times\Ff(U)\ar@<-.5ex>[r]_>>>>{id_{\Ff(V)}\times\cj} \ar@<.5ex>[r]^>>>>{pr_{13}} & \Ff(V)\times\Ff(U)\ar@<-.5ex>[r]^>>>>{\cj}\ar@/^2pc/[l]^{(\Delta_{\Ff(V)},id_{\Ff(U)})}&\Ff(U)\ar@/^2pc/[l]^{(|^U_V,id_{\Ff(U)})}}$$because for $(y_1,y_2,x)\in\Ff(V)\times\Ff(V)\times\Ff(U)$ we get $y_1\cj(y_2\cj x)=(y_1\cj y_2)\cj x=y_1\cj x$ so that $\cj\circ pr_{13}=\cj\circ (id_{\Ff(V)}\times\cj)$. Moreover, starting in the middle with $(y,x)\in\Ff(V)\times\Ff(U)$ the right hand idempotent gives $(y\cm(y\cj x),y\cj x)=(y,y\cj x)$ and the left hand idempotent gives the same result. Therefore, the two coequalisers are isomorphic showing that $\Ff(U)\cong\Ff(V)\times\Ff_{U\backslash V}(U)$ over $\Ff(V)$.\end{proof}

%%%%%%%%%%%%%%%%%%%%
\section{The patch monad}\label{sct:patch}
%%%%%%%%%%%%%%%%%%%%

This section contains the two main results of this article (Theorems \ref{thm:main1} and \ref{thm:main2}). These relate the category of algebras over the patch monad to two seemingly quite different categories, namely distributive skew lattices on one side, and saturated Priestley sheaves on the other. At the heart of these categorical equivalences is a \emph{monadicity theorem} for which we present two proofs.

%%%%%%%%%%%%%%%%%%%%
\subsection{The patch topology}\label{subsct:patch}\vspace*{1ex}
%%%%%%%%%%%%%%%%%%%%
Each spectral space $\Xx$ defines a Boolean space $\Xx^p$ with the same underlying set but a finer topology, the so-called \emph{patch topology}, cf. Hochster \cite{Hoch}. Under Stone duality, the patch refinement $\Xx^p\to \Xx$ corresponds to the Boolean envelope $\KO(\Xx)\hookrightarrow\KO(\Xx)_B$ of the dual lattice, cf. Section~\ref{subsct:stone}.

The existence of a patch topology relies first and foremost on the existence of a \emph{dual} topology on spectral spaces (reversing the specialisation order) since the patch topology coincides with the join of the spectral topology and its dual. This kind of order-reversing duality was introduced by de Groot \cite{deGroot67} and studied in more detail by Hochster \cite{Hoch} and in a more general setting by Lawson \cite{Lawson91}, see Kopperman \cite{Kopperman95} for a historical overview. In the algebraic geometer's tradition (cf. Grothendieck-Dieudonn\'e \cite{Groth}), the patch topology is called the \emph{constructible} topology.

Boolean spaces are \emph{zero-dimensional} compact Hausdorff spaces. They are often called ``Stone spaces'' or ``profinite sets'', cf. Johnstone \cite{Johnstone}. For a spectral space $\Xx$ the patch topology is most easily defined by taking as subbasis the compact open subsets of $\Xx$ together with their complements. Recall that a \emph{constructible} subset of a spectral space $\Xx$ is any element of the Boolean subalgebra generated by the compact open subsets of $\Xx$. The constructible subsets of a spectral space thus form a \emph{basis} for the patch topology. The constructible subsets of $\Xx$ are precisely the \emph{clopen subsets} of $\Xx^p$, i.e. those subsets which are simultaneously open and closed in $\Xx^p$. The \emph{spectral subspaces} of $\Xx$ are precisely the \emph{closed subsets} of $\Xx^p$, cf. Hochster \cite{Hoch}. This means that the constructible subsets of $\Xx$ can also be characterised as those subsets $Z$ for which $Z$ and $\Xx-Z$ are both spectral subspaces of $\Xx$.

Any spectral space $\Xx$ defines a \emph{specialisation order} on its points, namely $x\leq y$ if and only if $x\in\overline{\{y\}}$. The Boolean space $\Xx^p$ equipped with the specialisation order of $\Xx$ is known as the \emph{Priestley space} associated with $\Xx$. This ordered space satisfies Priestley's separation axiom of \emph{total order disconnectedness} \cite{P}: whenever $x\not\leq y$ there exists a clopen upset of $\Xx^p$ containing $x$ but not $y$.
When we use the notation $\Xx^p$ we always mean the Priestley space associated with $\Xx$, i.e. \emph{the specialisation order of $\Xx$ is supposed to be recorded}.
%\marginpar{\tiny MG. In what follows, you often use $\Xx^p$ for just the Boolean space. I think we need separate notations for this space and the Priestley space $(\Xx^p,\leq)$. Maybe we could denote the latter $\Xx^{pr}$? CB: I am not sure about this. Mostly it is the Priestley space we need ...}
The topology of the spectral space $\Xx$ can be recovered from the associated Priestley space $\Xx^p$ as the one admitting the \emph{clopen upsets} as basis and the \emph{open upsets} as opens. This actually defines an \emph{isomorphism} between the category of spectral spaces and spectral maps and the category of Priestley spaces and order-preserving continuous maps, cf. \cite{P}.

%%%%%%%%%%%%%%%%%%%%
\subsection{Distributive skew lattices as algebras over the patch monad}
%%%%%%%%%%%%%%%%%%%%
\
\vspace*{1ex}

\noindent
Let $\Xx$ be a spectral space. Consider the continuous map
\[
\phi:\Xx^p\to\Xx,
\]
given by the identity on the underlying set. This map induces an adjunction between sheaf categories
\[
\phi^*:\Sh(\Xx)\leftrightarrows\Sh(\Xx^p):\phi_*,
\]
where the inverse image functor $\phi^*$ and direct image functor $\phi_*$ are defined as follows. The direct image functor is given by restricting a sheaf $\Gg$ on $\Xx^p$ to a sheaf on the opens of the spectral space $\Xx$, which are open upsets of $\Xx^p$. On the other hand, the inverse image functor $\phi^*$ takes a sheaf $\Ff$ on $\Xx$ to the sheafification of the presheaf, given on clopens by $V\mapsto \colim_{U\supseteq V}\Ff(U)$ on $\Xx^p$. The colimit ranges over all compact open subsets $U$ of $\Xx$ containing the clopen subset $V$ of $\Xx^p$.

The composite functor
\[
\PM=\phi_*\phi^*:\Sh(\Xx)\rightarrow\Sh(\Xx).
\]
is then a \emph{monad} $(\PM,\eta,\mu)$ which we call the \emph{patch monad}. This monad is reminiscent of Godement's ``flabbyfication'' monad associated with the discretisation $\Xx^{disc}\to\Xx$ of a topological space, cf. Proposition~\ref{prp:skewsheaf}. In the following statement, global $T$-algebra means a global sheaf with $T$-algebra structure, and skew lattice over a given lattice $L$ means a skew lattice whose lattice reflection is $L$.

\begin{thm}\label{thm:main1}For any spectral space $\Xx$, the category of global $\PM$-algebras is equivalent to the category of right distributive skew lattices over the dual lattice $\KO(\Xx)$.\end{thm}

The proof of this theorem will occupy half of this section. We begin by explicitly constructing the patch monad $\PM$, before describing the two mutually quasi-inverse functors relating $\PM$-algebras and right distributive skew lattices. In view of Theorem~\ref{thm:main0}, spectral sheaves are equivalent to distributive bands, and by Corollary~\ref{cor:forgetful}, each distributive skew lattice has an underlying distributive band. Theorem~\ref{thm:main1} may be read as follows: ``skew structures'' on a right distributive band are in one-to-one correspondence with $\PM$-algebra structures on the dual spectral sheaf.

\subsection{Explicit construction of the patch monad}\label{sct:patchmonad}

It will be useful to describe the inverse image functor on the level of local sections. As stated above, $\phi^*\Ff$ is obtained as sheafification of the presheaf
\begin{align*}
\phi^\dagger \Ff\colon  \KO(\Xx^p)^\op\  \to \ &\Set\\
                            V    \     \mapsto \  &\left(\coprod_{U \in\KO(\Xx),U\supseteq V}\Ff(U)\right)/\sim_V
\end{align*}
where $\coprod$ denotes the disjoint union and the equivalence relation $\sim_V$ is given by
\[
(s,U)\sim_V(s',U')\iff \exists U''\in \KO(\Xx)\text{ with }V\subseteq U''\subseteq U\cap U'\text{ and } s|^U_{U''}= s'|^{U'}_{U''}.
\]
To see that $\phi^\dagger \Ff$ is in general only a presheaf, notice that $(\phi^\dagger \Ff)(V)$ can be identified with $\Ff({\uparrow} V)$ where ${\uparrow} V$ denotes the upset generated by $V$, which can be shown to be compact and equal to the intersection of all compact open subsets containing $V$.

Even for disjoint clopen subsets $V_1$ and $V_2$, the upsets ${\uparrow}V_1$ and ${\uparrow}V_2$ may not be disjoint. Thus, if we take arbitrary local sections $s_i\in \Ff({\uparrow}V_i)$, they may not be compatible and thus not have a common extension to  ${\uparrow}V_1\cup{\uparrow}V_2={\uparrow}(V_1\cup V_2)$. However, since  $V_1$ and $V_2$ are disjoint, as elements of the sheafification of $\phi^\dagger \Ff(V_i)$, they should have a common extension to $V_1\cup V_2$. Notice that $\phi^\dagger \Ff$ is however a \emph{separated} presheaf, since, whenever $s_i\in \Ff({\uparrow}V_i)$ have a common extension to ${\uparrow}(V_1\cup V_2)$, this extension is unique because $\Ff$ is a sheaf. Recall that the sheafification of a separated presheaf is a one-step construction while the sheafification of a general presheaf first constructs a separated presheaf before constructing a sheaf, cf. \cite{MoerMacL}.

The sheafification of $\phi^\dagger \Ff$ is the closure of $\phi^\dagger \Ff$ under \emph{formal finite disjoint patch} modulo the appropriate notion of refinement. Since we can refine any finite clopen cover to a finite clopen \emph{partition}, it suffices to consider clopen partitions, and for these the gluing condition is vacuous, thus simplifying the description.
%\marginpar{\tiny MG. It is my opinion that notation is eazed if we allow `partitions' that are allowed to contain the empty set. CB: yes, but I would not change terminology.}

Given $U\in\KO(\Xx)$, $s\in \Ff(U)$ and $V\in \KO(\Xx^p)$ with $U\supseteq V$, we write
$(s,U;V)$ for the \emph{equivalence class} in $\Ff({\uparrow}V)$ represented by $(s,U)$, as defined above. Now, for a clopen partition  $V_1\sqcup\dots\sqcup V_n$ of $V\in\KO(\Xx^p)$, $s_i\in \Ff(U_i)$ with $U_i\in\KO(\Xx)$, and $U_i\supseteq V_i$ for each $i=1,\dots, n$, we denote by
\[
\bigveedot_{i=1}^n (s_i,U_i;V_i) = (s_1,U_1;V_1)\veedot (s_2,U_2;V_2)\veedot\cdots\veedot (s_n,U_n;V_n)
\]
the formal patch of this family, which is patching since $V_1,\dots,V_n$ are disjoint. Taking all these formal patches as distinct does not yield a sheaf as we would lack uniqueness of gluing. However, if we identify two formal patches whenever they admit a \emph{common refinement}, then we get the sheafification $\phi^*\Ff$ of $\phi^\dagger \Ff$.

A formal patch $\inlinebigveedot_{j=1}^{n} (t_j,U'_j;V'_j)$ \emph{refines} the formal patch $\inlinebigveedot_{i=1}^m (s_i,U_i;V_i)$ provided the partition $(V'_j)_{j=1,\dots,n}$ refines the partition $(V_i)_{i=1,\dots,m}$, that is, for each $j$ there is $i$ such that $V'_j\subseteq V_i$, and moreover
 $(t_j,U'_j)\sim_{V'_j}(s_i,U_i)$. Note that the last property does not depend on the chosen representatives on both sides.

Later on, we adopt the convention that \emph{if a construction produces partitions with empty parts, then the latter will be dropped from the patch.} This convention is consistent with the fact that sheaves have singleton value on the empty set.%\marginpar{\tiny CB: Is this ok ?}

Since direct image $\phi_*$ is restriction and clopen subsets of $\Xx^p$ correspond to constructible subsets of $\Xx$, we get the following description of the \emph{patch monad}:
\begin{align*}
\PM \Ff:\KO(\Xx)^\op & \to \Set\\
             U &\mapsto \left(\bigcup_{n\geq 1}\{\bigveedot_{i=1}^n (s_i,U_i;V_i)\mid \{V_i\}_{i=1}^n\text{ constructible partition of }U\}\right)/\sim
\end{align*}
The unit of the patch monad is given by
\begin{align*}
\eta_\Ff\colon \Ff &\to \PM \Ff\\
                    (s,U)&\mapsto (s,U;U)
\end{align*}
and for the multiplication $\mu_\Ff\colon \PM ^2\Ff\to \PM \Ff$ let $U\in\KO(\Xx)$ and  $\inlinebigveedot_{i=1}^n (t_i,W_i;V_i)\in (\PM ^2\Ff)(U)$. That is, $\{V_i\}_{i=1}^n$ is a constructible partition of $U$, $V_i\subseteq W_i$, and $t_i\in(\PM \Ff)(W_i)$ for each $i=1,\dots,n$. Now $t_i\in (\PM \Ff)(W_i)$ means that
\[
t_i=\bigveedot_{j=1}^{n_i} (s_{ij},U_{ij};V_{ij})
\]
where $\{V_{i1},\dots,V_{in_i}\}$ is a constructible partition of $W_i$. It follows that $\{V_{ij}\cap V_i\mid 1\leq j\leq n_i, 1\leq i\leq n, \}$ is a constructible partition of $U$.
 Also, $V_{ij}\cap V_i\subseteq U_{ij}$ and
\[
^\cdot \hspace{-17.6pt}\bigvee_{\substack{1\leq j\leq n_i \\ 1\leq i\leq n}} (s_{ij},U_{ij};V_{ij}\cap V_i) \in (\PM \Ff)(U).
\]
The multiplication of the monad $\PM $ is then given by
\[
\mu_\Ff(\,\bigveedot_{i=1}^n (t_i,W_i;V_i))=\ \ ^\cdot \hspace{-17.6pt}\bigvee_{\substack{1\leq j\leq n_i \\ 1\leq i\leq n}} (s_{ij},U_{ij};V_{ij}\cap V_i)
\]
forgetting one level of partition.
%%%%%%%%%%%%%%%%%%%%
\subsection{From $\PM$-algebras to distributive skew lattices}\label{subsct:TtoDSL}
%%%%%%%%%%%%%%%%%%%%
Consider a $\PM$-algebra

\[
\xi_\Ff\colon \PM \Ff\to \Ff.
\]
By definition of a $\PM$-algebra we have $\xi_\Ff\circ\eta_\Ff=id_\Ff$ and $\xi_\Ff\circ \PM \xi_\Ff=\xi_\Ff\circ\mu_\Ff$.

By Proposition~\ref{prp:localsectionband} and Theorem~\ref{thm:main0}, any sheaf $\Ff$ on a spectral space $\Xx$ induces a right distributive band of local sections
\[
\el_{\KO(\Xx)}(\Ff)=\{(s,U)\mid s\in \Ff(U), U\in \KO(\Xx) \}
\]
with multiplication

\[
(s,U)(t,V)=(s,U)\cm(t,V)=(t|_{U\cap V},U\cap V)\in \Ff(U\cap V).\]
We claim that a $\PM$-algebra structure $\xi_\Ff:\PM \Ff\to \Ff$ allows us to extend the right distributive band $\el_{\KO(\Xx)}(\Ff)$ to a right distributive skew lattice $S_{\xi_\Ff}$ by introducing the following global join operation:
\[
(s,U)\cj(t,V):=\xi_\Ff((s,U;U)\veedot(t,V;V\backslash U))\in \Ff(U\cup V).
\]

\begin{prp}\label{prp:Tskew}
Let $\Ff$ be a global sheaf on the spectral space $\Xx$ and $\xi_\Ff\colon \PM \Ff\to \Ff$ be a $\PM$-algebra structure. Then $(\el_{\KO(\Xx)}(\Ff),\cm,\cj)$ is a right distributive skew lattice.\end{prp}

\begin{proof}
The semilattice reflection of $(\el_{\KO(\Xx)}(\Ff),\cm,\cj)$ may be identified with $\KO(\Xx)$ and is thus bounded distributive, cf. Proposition~\ref{prp:localsectionband}.

The unit constraint $\xi_\Ff\circ\eta_\Ff=\id_\Ff$ implies $\xi_\Ff(s,U;U)=(s,U)$ for any $U\in\KO(\Xx)$ and $s\in \Ff(U)$. The associativity constraint of $\xi_\Ff$ implies the associativity of $\cj$. Indeed, consider $U,V,W\in\KO(\Xx)$ and $r\in \Ff(U)$, $s\in \Ff(V)$ and $t\in \Ff(W)$. Then

\begin{enumerate}
\item[$\bullet$] $\{U,V\backslash U\}$ is a constructible partition of $U\cup V$ and thus
\[
\sigma=(r,U;U)\veedot(s,V;V\backslash U)\in (\PM \Ff)(U\cup V);
\]
\hskip 1.8cm $\tau=(t,W;W)\in(\PM \Ff)(W)$;\\
\item[$\bullet$]  $\{U\cup V,W\backslash(U\cup V)\}$ is a constructible partition of $U\cup V\cup W$ and thus
\[
(\sigma,U\cup V;U\cup V)\veedot(\tau,W;W\backslash(U\cup V))\in(\PM ^2\Ff)(U\cup V\cup W).
\]
\end{enumerate}

\noindent Applying $\PM \xi_\Ff:\PM ^2\Ff\to \PM \Ff$ to this element of $(\PM ^2\Ff)(U\cup V\cup W)$ yields\[(\xi_\Ff(\sigma);U\cup V)\veedot(\xi_\Ff(\tau);W\backslash(U\cup V))\in(\PM \Ff)(U\cup V\cup W)\]so that
the associativity constraint $\xi_\Ff\circ \PM \xi_\Ff=\xi_\Ff\circ\mu_\Ff$ implies
\[
\xi_\Ff((\xi_\Ff(\sigma);U\cup V){\sveedot} (\xi_\Ff(\tau);W\backslash(U\cup V))=\xi_\Ff(\mu_\Ff((\sigma,U\cup V;U\cup V)\sveedot(\tau,W;W\backslash(U\cup V))),
\]
from which in turn follows that
\begin{align*}
((r,U)\cj(s,V))\cj(t,W)&=\xi_\Ff((\xi_\Ff((r,U;U)\veedot(s,V;V\backslash U)); U\cup V)\veedot(t,W;W\backslash(U\cup V)))\\
   &=\xi_\Ff((\xi_\Ff(\sigma);U\cup V) \veedot (\xi_\Ff(\tau);W\backslash(U\cup V)))\\
   &=\xi_\Ff(\mu_{\Ff}((\sigma,U\cup V;U\cup V)\veedot(\tau,W;W\backslash(U\cup V)))\\
   &=\xi_\Ff((r,U;U)\veedot(s,V;V\backslash U)\veedot(t,W;W\backslash(U\cup V)))
\end{align*}
The last term involves a threefold patch in $\PM \Ff$ which is associative because $\PM \Ff$ is a sheaf. A similar argument shows that $(r,U)\cj((s,V)\cj(t,W))$ equals the image by $\xi_\Ff$ of the same threefold patch, thus proving associativity of $\cj$.

The four absorption laws are also established in a straightforward way:
\begin{align*}
((s,U)\cm(t,V))\cj(t,V)&= (t|_{U\cap V}, U\cap V)\cj(t,V)\\
                                          &=\xi_\Ff((t|_{U\cap V},U\cap V; U\cap V)\veedot(t,V;V\backslash(U\cap V))\\
                                          &=\xi_\Ff((t,V;V))= (t,V)
\end{align*}
and
\begin{align*}
(s,U)\cj((s,U)\cm(t,V))&= (s,U)\cj(t|_{U\cap V}, U\cap V)\hskip1.2cm\\
                                          &=\xi_\Ff((s,U;U)\veedot (t|_{U\cap V}, U\cap V;\emptyset))\\
                                          &=\xi_\Ff(s,U;U)=(s,U).
                                          \end{align*}
Also, since $\cm$ is right restriction, we have
\begin{align*}
((s,U)\cj(t,V))\cm(t,V)= (t,V)|_{(U\cup V)\cap U}=(t,V)
\end{align*}
and
\begin{align*}(s,U)\cm((s,U)\cj(t,V))
                    &=((s,U)\cj(t,V))|_U\\
                    &=\xi_\Ff((s,U;U)\veedot(t,V;V\backslash U))|_U\\
                    &=\xi_\Ff((s,U;U)\veedot(t,V;\emptyset))\\
                    &=\xi_\Ff((s,U;U))=(s,U).
\end{align*}

All that remains to be shown is the symmetry of $(\el_{\KO(\Xx)}(\Ff),\cm,\cj)$. That is, for local sections $(s,U),(t,V)$ of the spectral sheaf $\Ff$, we have to show that
\begin{align*}
&(s,U)\cm(t,V)=(t,V)\cm(s,U)\qquad(\cm\text{-commutativity}) \\
\quad\text{if and only if}\quad &(s,U)\cj(t,V)=(t,V)\cj(s,U) \qquad (\cj\text{-commutativity}).
\end{align*}
Note that $\cm$-commutativity holds if and only if $s|_{U\cap V}=t|_{U\cap V}$  if and only if
\[
(s, U; U\cap V)=(t, V; V\cap U)
\]
in $\PM \Ff$. Note also that
\begin{align*}
(s,U)\cj(t,V)&=\xi_\Ff((s,U;U)\veedot(t,V;V\backslash U))\\
                     &=\xi_\Ff((s,U;U\backslash V)\veedot(s,U;U\cap V)\veedot(t,V;V\backslash U))
\end{align*}
and
\begin{align*}
(t,V)\cj(s,U)&=\xi_\Ff((t,V;V)\veedot(s,U;U\backslash V))\\
                     &=\xi_\Ff((t,V;V\backslash U)\veedot(t,V;V\cap U)\veedot(s,U;U\backslash V)).
\end{align*}
It is thus clear that $(\el_{\KO(\Xx)}(\Ff),\cm,\cj)$ is a symmetric skew lattice as claimed.
\end{proof}
%%%%%%%%%%%%%%%%%%%%
\subsection{From distributive skew lattices to $\PM$-algebras}\label{subsct:DSLtoT}
%%%%%%%%%%%%%%%%%%%%

We start with a fundamental refinement property of constructible partitions of spectral spaces. Under Stone duality these correspond to Boolean partitions of bounded distributive lattices so that the refinement property can be formulated in both contexts. For the convenience of the reader we give individual proofs of the two dual statements.

We introduce the following terminology: a \emph{chain partition} of $\Xx$ is any constructible partition $(U_i)_{i=1,\dots,n}$ associated with an ascending chain of compact opens $$\emptyset=\Xx_0\subset \Xx_1\subset\cdots\subset \Xx_{n-1}\subset \Xx_n=\Xx$$ by letting $U_i=\Xx_i-\Xx_{i-1}$ for $i=1,\dots,n.$ Observe that, by definition, a chain partition is linearly ordered, while a general constructible partition is not.

\begin{lma}\label{lma:diffchainfin}Any constructible partition of a spectral space $\Xx$ can be refined to a chain partition.\end{lma}

\begin{proof}Let $(U_i)_{i=1,\dots,n}$ be a constructible partition of $\Xx$. Each constructible subset of $\Xx$ can be partitionned into constructible subsets of the form $V_i-V'_i$ for a pair of compact opens $V'_i\subset V_i$. We can thus assume without loss of generality that the individual $U_i$ have already this form $V_i-V'_i$. Consequently, $V'_i$ is contained in the union $U_1\sqcup\cdots\sqcup U_{i-1}\sqcup U_{i+1}\sqcup\cdots\sqcup U_n$ and we get a constructible partition$$V'_i=(U_1\cap V'_i)\sqcup\cdots\sqcup(U_{i-1}\cap V'_i)\sqcup(U_{i+1}\cap V'_i)\sqcup\cdots\sqcup(U_n\cap V'_i)$$in which each of the $n-1$ members is again the difference of two compact open subsets of $\Xx$. We can thus make an induction on $n$ and assume that the latter refines to a chain partition. By adding $U_i$ as last element we get a chain partition of $V_i$. Each member of this chain partition of $V_i$ is contained in some $U_j$.

The thus constructed chains $\emptyset=V_{0,i}\subset\cdots\subset V_{n(i),i}=V_i$ for the individual $V_i$ assemble into a chain for $V_1\cup\cdots\cup V_n=\Xx$ by taking intersections $V_{i_1,1}\cap\cdots\cap V_{i_n,n}$ in lexicographical order. The total chain has then the property that each difference of two successive chain members is contained in some $U_j$. In other words, the associated chain partition refines the given partition of $\Xx$.\end{proof}

We shall now formulate and prove the Stone-dual of Lemma~\ref{lma:diffchainfin}. Let $L$ be a bounded distributive lattice. A \emph{Boolean partition} of $L$ is a finite collection $\{b_1,\dots,b_n\}$ of elements of the \emph{Boolean envelope} $L_B$ of $L$ such that
\begin{enumerate}
\item $b_i\wedge b_j=0$ whenever $i\neq j$;
\item $\bigvee_{i=1}^nb_i=1$.
\end{enumerate}

Such a Boolean partition is called a \emph{Boolean chain partition} if there is a sequence $0=a_0<a_1<\dots<a_n=1$ of elements of $L$ such that $a_i-a_{i-1}=b_i$ for $1\leq i\leq n$.

\begin{lma}\label{diffchainfinbis}Each Boolean partition of a bounded distributive lattice can be refined to a Boolean chain partition.\end{lma}

\begin{proof}
Each element of a Boolean partition $b_1\vee\cdots\vee b_n$ of $L$ is a finite Boolean term in elements of $L$. There is thus a \emph{finite} bounded distributive sublattice $K$ of $L$ such that all $b_i$ belong to $K_B$ viewed as a Boolean sublattice of $L_B$.

The spectrum of a finite distributive lattice is a finite spectral space, and the topology of a finite spectral space is completely determined by its specialisation order. Therefore, in this finitary context, Stone duality amounts to \emph{Birkhoff's representation theorem} for finite distributive lattices. Using the latter we are reduced to showing that for any finite partially ordered set $P$, and any set-theoretical partition $P=Q_1\sqcup\cdots\sqcup Q_n$, there exists a chain of upsets $\emptyset=P_0\subset\cdots\subset P_m=P$ such that each successive difference $P_j-P_{j-1}$ is contained in one of the $Q_i$.

Such a chain of upsets can be constructed inductively. Suppose $P_0\subset\cdots\subset P_{k-1}$ with the required properties has already been constructed. If $P=P_{k-1}$ we are done. Otherwise, pick $p\in P-P_{k-1}$ maximal with respect to the induced partial order on $P-P_{k-1}$. There is a unique $i$ such that $p\in Q_i$. Then put $P_k=P_{k-1}\cup ({\uparrow}p)$. This is an upset of $P$ such that $P_k-P_{k-1}=\{p\}\subset Q_i$. The inductive construction terminates after a finite number of iterations since $P$ is finite.\end{proof}

\begin{rmk}It is an interesting problem to seek for chain refinements of minimal length. In the proof of Lemma~\ref{diffchainfinbis}, the initial input is the number of elements of $L$ needed in the Boolean expressions for the individual $b_i$. See \cite{BGKS} for an inductive construction of ``difference chains'' of minimal length for the elements of $L_B$.\end{rmk}

Given two constructible partitions $(U_i)_{i\in\{1,\dots,m\}}$ and $(V_j)_{j\in\{1,\dots,n\}}$ of a spectral space $\Xx$, if $(U_i)$ refines $(V_j)$, then there is a unique (surjective) \emph{choice function} $f:\{1,\dots,m\}\to\{1,\dots,n\}$ such that $U_{i}\subset V_{f(i)}$ for $i=1,\dots,m$. If $(U_i)$ and $(V_j)$ are chain partitions, we say that $(U_i)$ \emph{chain refines} $(V_j)$ if $(U_i)$ refines $(V_j)$ and the associated choice function is order-preserving. This is the case if and only if the defining chain of the former partition refines the defining chain of the latter.

A \emph{formal chain patch} of a sheaf $\Ff$ on $\Xx$ is a formal patch $\inlinebigveedot_{i=1}^m (s_i,U'_i;U_i)$ in the sense of Section~\ref{sct:patchmonad} such that the underlying partition $(U_i)$ is a chain partition.

\begin{lma}\label{lma:chainrefine}Any two chain partitions $(U_i)$ and $(V_j)$ have a common chain refinement $(W_{ij})$ where $W_{ij}=U_i\cap V_j$ and empty terms are discarded.\end{lma}

\begin{proof}Take all binary intersections of the members of the two defining chains and order them by inclusion. This defines a chain of compact opens whose chain partition $(W_{ij})$ chain refines $(U_i)$ and $(V_j)$.\end{proof}

We will say that a choice of representatives $s_i\in\Ff(U'_i)$ for the formal chain patch $\inlinebigveedot_{i=1}^m (s_i,U'_i;U_i)$ is \emph{admissible} if $U'_1\cup U'_2\cup\cdots\cup U'_i= U_1\sqcup U_2\sqcup\cdots\sqcup U_i$ for all $i=1,\dots,m$. So, admissibility concerns uniquely the domain of definition of the local sections. In the case of a formal \emph{chain} patch such an admissible choice of the domains always exists since the partial unions $U_1\sqcup U_2\sqcup\cdots\sqcup U_i$ are compact open subsets of $\Xx$ so that the local sections can be restricted if necessary.

%\marginpar{\tiny CB: With the corrected notion of patch refinement in Section \ref{sct:patchmonad}, the proof of Lemma \ref{lma:Talg} gets more involved. Please check.}

\begin{lma}\label{lma:Talg}Let $\Ff$ be a global sheaf on a spectral space $\Xx$ such that the right distributive band $\el_{\KO(\Xx)}(\Ff)$ extends to a right distributive skew lattice $S_\Ff$.
\begin{enumerate}\item[(a)]For a formal chain patch $\inlinebigveedot_{j=1}^n (t_j,V'_j;V_j)$ of $\Ff$, the join $t_1\cj\cdots\cj t_n$ in $S_\Ff$ is independent of any admissible choice of representatives $t_j\in\Ff(V'_j)$.
\item[(b)]Let $\inlinebigveedot_{i=1}^m (s_i,U'_i;U_i)$ and $\inlinebigveedot_{j=1}^n (t_j,V'_j;V_j)$ be formal chain patches of $\Ff$ such that $U_1\sqcup\cdots\sqcup U_m=V_1\sqcup\cdots\sqcup V_n$ and for any non-empty intersection $U_i\cap V_j$ we have $(s_i,U'_i)\sim_{U_i\cap V_j}(t_j,V'_j)$. Then for any admissible representatives on both sides we get the same joins $s_1\cj\cdots\cj s_m=t_1\cj\cdots\cj t_n$ in $S_\Ff$.\end{enumerate}\end{lma}

\begin{proof}(a) Let $\inlinebigveedot_{j=1}^n (t_j,V'_j;V_j)$ and $\inlinebigveedot_{j=1}^n (\tilde{t}_j,\tilde{V}'_j;V_j)$ be two admissible representations of the same formal chain patch of $\Ff$. We argue by induction on $j$. For $j=1$, necessarily $V_1=V'_1$ so that $t_1=\tilde{t}_1$. Suppose inductively that $t=t_1\cj\cdots\cj t_{j-1}=\tilde{t}_1\cj\cdots\cj\tilde{t}_{j-1}=\tilde{t}$. We want to show that $t_j\sim_{V_j}\tilde{t}_j$ implies $t\cj t_j=\tilde{t}\cj \tilde{t}_j$.

By definition of $V_j$-equivalence there is a local section $u\in\Ff(W)$ such that $V_j\subset W$ and $u\leq t_j$ and $u\leq\tilde{t}_j$. Since $S_\Ff$ is a right-handed skew lattice and hence $(S_\Ff,\cj)$ a left regular band, we get $t\cj u\leq t\cj t_j$ and $t\cj u\leq t\cj \tilde{t}_j$. Now, the admissibility of the two formal patch representations implies that the $\Dd$-classes (i.e. domains of definition) of the latter three local sections are the same. It follows then from Proposition \ref{prp:basic}iv that $t\cj t_j=t\cj\tilde{t}_j$ as required.

(b) By Lemma~\ref{lma:chainrefine} there is a common chain refinement $(W_{ij})$ of $(U_i)$ and $(V_j)$. Because of the condition $(s_i,U'_i)\sim_{U_i\cap V_j}(t_j,V'_j)$ there is a local section $u_{ij}\in\Ff(W'_{ij})$ which is a common restriction of the local sections $s_i$ and $t_j$. Therefore we can assume without loss of generality that $(U_i)$ already chain refines $(V_j)$ and that for admissible representatives $(s_i,U'_i),(t_j,V'_j)$, an inclusion $U_i\subset V_j$ implies the inclusion $U'_i\subset V'_j$ as well as $s_i=(t_j)_{|U'_i}$. This translates into $s_i=s_i\cm t_j$ whenever $f(i)=j$, where $f$ is the choice function associated with the chain refinement.

Let $f^{-1}(j)=\{i_1,\dots,i_{m(j)}\}$. This is a subinterval of $\{1,\dots,m\}$, and for varying $j$ we get an \emph{ordered} partition of $\{1,\dots,m\}$ since $f$ is order-preserving and surjective. Therefore, we can reduce (b) to (a) by showing that for each $j$ we have $V_j$-equivalent local sections $(s_{i_1}\cj\cdots\cj s_{i_{m(j)}},U'_{i_1}\cup\cdots\cup U'_{i_{m(j)}})\sim_{V_j}(t_j,V'_j)$.

Using distributivity and right-handedness of $S_\Ff$ we get $$s_{i_1}\cj\cdots\cj s_{i_{m(j)}}=(s_{i_1}\cm t_j)\cj\cdots\cj(s_{i_{m(j)}}\cm t_j)=(s_{i_1}\cj\cdots\cj s_{i_{m(j)}})\cm t_j$$ whence $s_{i_1}\cj\cdots\cj s_{i_{m(j)}}\leq t_j$ which implies the required $V_j$-equivalence.\end{proof}

It is remarkable that the proof of Lemma \ref{lma:Talg} only uses part of the structure of right distributive skew lattice of $S_\Ff$, namely its right normality and the distributivity of the lattice reflection $S_\Ff/\Dd$, cf. Lemma \ref{lma:normal}. Just below we show that Lemma \ref{lma:Talg} implies the existence of a canonical action map $\xi_\Ff\colon\PM\Ff\to\Ff$. It is the naturality of $\xi_\Ff$ which requires the entire right distributive skew structure of $S_\Ff$.

\begin{prp}Let $\Ff$ be a global sheaf on a spectral space $\Xx$ such that the right distributive band $\el_{\KO(\Xx)}(\Ff)$ extends to a right distributive skew lattice $S_\Ff$.

Then the map $\xi_\Ff\colon \PM \Ff\to \Ff$ taking an admissibly represented formal chain patch  $\inlinebigveedot_{i=1}^m (s_i,U'_i;U_i)\in (\PM \Ff)(U_1\sqcup\cdots\sqcup U_m)$ to $s_1\cj\cdots\cj s_m\in \Ff(U_1\sqcup\cdots\sqcup U_m)$ defines a $\PM$-algebra structure on $\Ff$.\end{prp}

\begin{proof}We first show that it is enough to specify the values of $\xi_F$ for admissibly represented formal chain patches. Indeed, by Lemma \ref{lma:diffchainfin}, for every formal patch $\inlinebigveedot_{k=1}^l (u_k,W'_k;W_k)$, the constructible partition $(W_k)$ admits a chain refinement $(U_i)$, inducing thus a formal chain patch $\inlinebigveedot_{i=1}^m (s_i,U'_i;U_i)$ with representatives $s_i=u_{f(i)}$ where $f$ is the choice function of the refinement. These representatives can be made admissible by suitable restrictions.

We next show that the values of $\xi_F$ do not depend on the choice of such formal chain patch refinements. Let $\inlinebigveedot_{i=1}^m (s_i,U'_i;U_i)$ and $\inlinebigveedot_{j=1}^n (t_j,V'_j;V_j)$ be two formal chain patches refining $\inlinebigveedot_{k=1}^l (t_j,W'_j;W_j)$. We have to show that the condition of Lemma~\ref{lma:Talg}(b) is satisfied. Indeed, if $U_i\cap V_j$ is non-empty then $U_i\cap V_j\subset W_k$ for a unique $k$. By construction, the local representatives $s_i\in\Ff(U'_i)$ and $t_j\in\Ff(V'_j)$ are restrictions of $u_k\in\Ff(W_k)$. We therefore get $(s_i,U')\sim_{U_i\cap V_j}(t_j,V'_j)$ as required.

Consequently, $\xi_\Ff\colon \PM \Ff\to \Ff$ is well-defined. Proposition~\ref{prp:skewsheaf} implies naturality of $\xi_\Ff$. Unit and associativity constraints of the $T$-algebra action map $\xi_\Ff$ are straightforward verifications.\end{proof}

\subsection{Proof of Theorem~\ref{thm:main1}}

The constructions of Sections~\ref{subsct:TtoDSL} and~\ref{subsct:DSLtoT} are functorial. In order to show that they are mutually quasi-inverse, it suffices to observe that, given a distributive band $(S,\cm)$ and its dual spectral sheaf $\Ff:(S/\Dd)^\op\to\Set$, the two constructions set up a bijection between global join operations on $(S,\cm)$ (turning the latter into a right distributive skew lattice) and global $\PM$-algebra structures on $\Ff$. It remains to be shown that the respective notions of morphism correspond to each other under this bijection, which is also straightforward.\qed

%%%%%%%%%%%%%%%%%%%%
\subsection{Exact categories and monadic right adjoints}
%%%%%%%%%%%%%%%%%%%%

Any adjunction $$L:\CC\leftrightarrows\DD:R$$ with left adjoint $L$ and right adjoint $R$ induces a monad $(RL,\eta, R\epsilon L)$ on $\CC$ where $\eta$ (resp. $\epsilon$) denotes unit (resp. counit) of the adjunction. We also have a canonical \emph{comparison functor} $k:\DD\to \CC^{RL}$ taking an object $D$ of $\DD$ to the $RL$-algebra $(RD,R\epsilon_D)$. The right adjoint $R$ is called \emph{monadic} if $k:\DD\to \CC^{RL}$ is an equivalence of categories so that $\DD$ may be identified with the category of $RL$-algebras on $\CC$.

We would like to establish a monadicity theorem involving the patch monad $T=\phi_*\phi^*$. There is a ``support-theoretical'' obstruction preventing the direct image functor $\phi_*:\Sh(\Xx^p)\to\Sh(\Xx)$ from being monadic, which forces us to restrict $\phi_*$ to the subcategory $\Sh_{s}(\Xx^p)\into\Sh(\Xx^p)$ consisting of sheaves with \emph{saturated support}. The latter category has not anymore all properties of a general category of sheaves (i.e. is not anymore a Grothendieck topos) but continues to be an \emph{exact} category.

The patch adjunction restricts to $\phi^*:\Sh(\Xx)\leftrightarrows\Sh_{s}(\Xx^p):\phi_*$ and we will show in Theorem~\ref{thm:main2} below that the restricted direct image functor is monadic and thus induces an equivalence of categories $k:\Sh_{s}(\Xx^p)\eqv\Sh(\Xx)^\PM$.

Two proofs of Theorem~\ref{thm:main2} will be given: an ``abstract'' proof based on Duskin's monadicity criterion \cite{D}, and a ``constructive'' proof by exhibiting a quasi-inverse $l:\Sh(\Xx)^\PM\eqv\Sh_{s}(\Xx^p)$ to $k$, see Definition~\ref{dfn:constructive} and Propositions \ref{prp:Gsheaf} and \ref{prp:Gsheaf2}.

A category is called \emph{exact} (in the sense of Barr \cite{Barr}) if all finite limits exist, equivalence relations are effective\footnote{this means that if $R\dto X$ represents an equivalence relation then its coequaliser $q:X\to X/R$ exists and $R\dto X$ is the kernel pair of $q$.}, and regular epimorphisms (i.e. coequalisers) are stable under pullback. Important examples of exact categories are:\begin{itemize}\item Abelian categories\item categories of sheaves on a topological space (or Grothendieck site)\item varieties of $T$-algebras for a monad $T$ on sets\end{itemize}In the first two examples all epimorphisms are regular, in the third example regular epimorphisms are those $T$-algebra maps whose underlying set mapping is surjective. The category of compact Hausdorff spaces and continuous maps is exact (since compact Hausdorff spaces are the algebras for the ultrafilter monad on sets) while the category of all topological spaces and continuous maps is not (because regular epimorphisms are not stable under pullback).

A functor between exact categories is called \emph{exact} if it preserves finite limits and regular epimorphisms. Since in an exact category, every equivalence relation is the kernel pair of its coequaliser, and every regular epimorphism is the coequaliser of its kernel pair, a finite limit preserving functor between exact categories is exact if and only if it preserves quotients of equivalence relations.

\begin{prp}\label{prp:monadic}Let $R$ be a right adjoint functor between exact categories. If $R$ reflects isomorphisms and preserves regular epimorphisms then $R$ is monadic.\end{prp}

\begin{proof}Since $R$ is right adjoint, $R$ preserves limits, whence $R$ is exact. Since $R$ reflects isomorphisms, and $R$ preserves quotients of equivalence relations, $R$ also reflects them. This together with the effectiveness of equivalence relations implies that $R$ \emph{creates} quotients of equivalence relations. This in turn yields monadicity of $R$ according to Duskin \cite[Thm. 3.0]{D}, cf. also Barr-Wells \cite[Prop. 9.1.8]{BarrWells}.\end{proof}

Recall that a \emph{coreflective} subcategory is a full subcategory such that the inclusion has a right adjoint, cf. Remark~\ref{rmk:reflective} for the dual notion of a reflective subcategory.

\begin{prp}\label{prp:exactsubcat}Let $\CC$ be a \emph{coreflective} subcategory of an exact category $\DD$ such that the inclusion preserves finite limits, and for each regular epimorphism $f$ in $\DD$, the domain of $f$ belongs to $\CC$ if and only if the codomain of $f$ belongs to $\CC$.

Then $\CC$ is an exact category, and the inclusion $\CC\into\DD$ is an exact functor.\end{prp}

\begin{proof}Since $\CC$ is a coreflective subcategory of $\DD$ and $\DD$ has all finite limits, $\CC$ also has all finite limits. By hypothesis, the latter are preserved by the inclusion. In particular, an equivalence relation in $\CC$ is also an equivalence relation in $\DD$ and hence possesses a quotient in $\DD$ which belongs to $\CC$ by the second hypothesis. As a left adjoint, the inclusion also preserves regular epimorphisms. The pullback of a regular epimorphism in $\CC$ yields thus a regular epimorphism in $\DD$ which (again by the second hypothesis) belongs to $\CC$ and is a regular epimorphism in $\CC$. Therefore, $\CC$ is an exact category, and the inclusion $\CC\into\DD$ is an exact functor.\end{proof}

%In an exact category regular epimorphisms compose, and each morphism factors essentially uniquely as a regular epimorphism followed by a monomorphism. Because of this, the right adjoint $R:\DD\to\CC$ reflects isomorphism if and only if the counit $\epsilon_D:RLD\to D$ is a regular epimorphism for all objects $D$ of $\DD$.

\begin{thm}\label{thm:main2}For any spectral space $\Xx$, the restricted direct image functor $\phi_*:\Sh_{s}(\Xx^p)\to\Sh(\Xx)$ is monadic so that the category of sheaves with saturated support on $\,\Xx^p$ is equivalent to the category of $\,\PM$-algebraic sheaves on $\,\Xx$.

In particular, the category of $\,\PM$-algebras is an exact category.\end{thm}

\begin{proof}The support of a spectral sheaf is an upset for the specialisation order (i.e. \emph{saturated}) and the inverse image functor $\phi^*:\Sh(\Xx)\to\Sh(\Xx^p)$ preserves supports. Therefore $\phi^*$ factors through $\Sh_{s}(\Xx^p)$ and we get a restricted adjunction $$\phi^*:\Sh(\Xx)\lra\Sh_{s}(\Xx^p):\phi_*.$$

We first show that $\Sh_{s}(\Xx^p)$ is a \emph{coreflective} subcategory of $\Sh(\Xx^p)$. The support of a sheaf can be defined using the epi/mono factorisation and the bijection between open subsets of $\Xx^p$ and subobjects of the terminal sheaf $\one$. Indeed, under this bijection, the support of a sheaf $\Gg$ may be identified with the image of the unique map $\Gg\to\one$, i.e. $\Gg\onto\supp(\Gg)\into\one$. The subcategory $\Sh_s(\Xx^p)$ consists of those sheaves whose support is an open upset for the specialisation order of $\Xx$. This condition on the support of $\Gg$ amounts to requiring that the ``support-counit'' $\phi^*\phi_*\supp(\Gg)\to\supp(\Gg)$ is invertible. For a general sheaf $\Gg$, the coreflection into $\Sh_s(\Xx^p)$ is then defined by restricting $\Gg$ to the maximal open upset of its support. Formally, this is the pullback of $\Gg\onto\supp(\Gg)$ along the support-counit $\phi^*\phi_*\supp(\Gg)\to\supp(\Gg)$. The universal property of the coreflection follows from the fact that $\phi^*\phi_*$ is idempotent on subobjects of terminal objects.

We shall now apply Proposition~\ref{prp:exactsubcat} in order to show that $\Sh_s(\Xx^p)$ is an \emph{exact} subcategory of $\Sh(\Xx^p)$. For any epimorphism in $\Sh(\Xx^p)$ domain and codomain have the same support so that the second hypothesis of Proposition~\ref{prp:exactsubcat} is satisfied. In order to show that the inclusion $\Sh_s(\Xx^p)\into\Sh(\Xx^p)$ preserves finite limits, it is enough to show that it preserves terminal object and pullbacks.

The terminal sheaf on $\Xx^p$ has saturated support and belongs thus to $\Sh_s(\Xx^p)$. Since the epi/mono factorisation of $\Sh(\Xx^p)$ restricts to $\Sh_s(\Xx^p)$, a pullback in $\Sh_s(\Xx^p)$ can be decomposed into pullbacks having either both legs monomorphic or at least one leg epimorphic. In the latter case, the second already established hypothesis of Proposition~\ref{prp:exactsubcat} implies that these pullbacks are preserved by the inclusion; in the former case, the pullbacks are preserved as well because the support of an intersection of two subsheaves equals the intersection of the individual supports, and the intersection of two open upsets is again an open upset.

We are now able to apply Proposition~\ref{prp:monadic} to the restricted adjunction. By a general result, the right adjoint $\phi_*$ reflects isomorphisms if and only if the counit $\phi^*\phi_*\Gg\to \Gg$ is a regular epimorphism in $\Sh_s(\Xx^p)$ for all sheaves $\Gg$ with saturated support. A morphism of set-valued sheaves $f:\Gg'\to \Gg$ is a regular epimorphism if and only if it an epimorphism if ad only if the induced morphisms on stalks $f_x:\Gg'_x\to \Gg_x$ are surjective, i.e. for each $x\in\Xx^p$, each open neighborhood $V$ of $x$, and each local section $t\in \Gg(V)$, there is an open neighborhood $U\subset V$ of $x$ and a local section $s\in \Gg'(U)$ such that $f_U(s)=t|_U$.

Since $\Gg'=\phi^*\phi_*\Gg$ and $f$ is the counit, this amounts to showing that each section of $\Gg$ over a sufficiently small open neighborhood of $x$ extends to a section over a saturated open neighborhood. For this, note that the topology of $\Xx^p$ is generated by the collection of intersections $U\cap V$ where $U$ is a clopen downset and $V$ is a clopen upset. Since $x\in\supp(\Gg)$ we can assume without loss of generality that $x\in U\cap V$ and $V\subset\supp(\Gg)$. Now, $(U\cap V)\sqcup(V- U\cap V)$ is a clopen partition of $V$, and $\Gg(V -U\cap V)$ is non-empty, cf. Lemma~\ref{lma:g=g}. The gluing property of the sheaf $\Gg$ implies then that every section over $U\cap V$ extends to a section over $V$.

For the second hypothesis of Proposition~\ref{prp:monadic}, let us assume that $f:\Gg'\to \Gg$ is a regular epimorphism in $\Sh_s(\Xx^p)$. We have to show that $\phi_*f:\phi_*\Gg'\to\phi_*\Gg$ is an epimorphism in $\Sh(\Xx)$. Let $U$ be an open neighborhood of a point $x$ in the spectral space $\Xx^p$ and $t\in (f_*\Gg)(V)$ a local section. Since the topology of the spectral space $\Xx$ is generated by compact open subsets we can assume that $V$ is compact open in $\Xx$. Then, for each $y\in V$, there exists a clopen neighborhood $U_y$ in $\Xx^p$ and a section $s_y\in \Gg'(U_y)$ such that $f_{U_y}(s_y)=t|_{U_y}$ because $f:\Gg'\to \Gg$ is an epimorphism in $\Sh(\Xx^p)$. Since $V$ is compact, there exists a finite number of $U_y$'s covering $V$. Take all non-empty finite intersections of these finitely many $U_y$'s: they form a \emph{clopen partition} of $V$ in $\Xx^p$ such that over each individual member there is a local section of $\Gg'$ mapped to the corresponding restriction of $t$ by $f$. The gluing property of the sheaf $\Gg'$ ensures the existence of a section $s\in(\phi_*\Gg')(V)$ such that $(\phi_*f)_V(s)=t$, whence $\phi_*f$ is an epimorphism in $\Sh(\Xx)$.\end{proof}

\begin{exm}Let us illustrate Theorem~\ref{thm:main2} for the Sierpinski spectral space $\Xx=\{0,1\}$ consisting of two points with $\{0\}$ closed and $\{1\}$ open. The associated Priestley space $\Xx^p$ is the discrete space $\{0,1\}$ equipped with the natural order.

A sheaf $\Gg$ on the discrete Priestley space $\Xx^p$ is determined by its stalks $(\Gg_0,\Gg_1)$ and has saturated support whenever $\Gg_0\not=\emptyset$ implies $\Gg_1\not=\emptyset$.

A sheaf $\Ff$ on the Sierpinski space $\Xx$ is given by its set of global sections $\Ff(\{0,1\})$ together with a restriction map $\Ff(\{0,1\})\to \Ff(\{1\})$. Such a sheaf is a $T$-algebra if and only if it has a skew distributive lattice of local sections, cf. Theorem~\ref{thm:main1}, if and only if the set $\Ff(\{0,1\})$ of global sections may be identified (over $\Ff(\{1\})$) with the product $\Ff(\{1\})\times\Ff_{\{0\}}(\{0,1\})$, cf. Proposition \ref{prp:skewsheaf}iii.

The comparison functor $k:\Sh_s(\Xx^p)\to\Sh(\Xx)^T$ induces an equivalence from the former structure to the latter structure.\end{exm}

\begin{lma}\label{lma:g=g}Any globally supported sheaf on a Boolean space has global sections.\end{lma}

\begin{proof}By hypothesis there is a partition of the whole space into clopen subsets supporting local sections. The latter glue to give a global section.\end{proof}

\begin{thm}[\cite{BCGGK}]\label{thm:main3}The category of globally supported sheaves on Priestley spaces is dually equivalent to the category of right distributive skew lattices.\end{thm}
\begin{proof}This follows by combining the preceding lemma with Theorems~\ref{thm:main1},~\ref{thm:main2} and Priestley duality \cite{P}.\end{proof}

\subsection{A constructive proof of Theorem~\ref{thm:main2}}\label{subsct:constructive}--\vspace{1ex}

We fix a sheaf $\Ff\colon \KO(\Xx)^\op\to\Set$ on a spectral space $\Xx$ and a $\PM$-algebra structure $\xi_\Ff\colon \PM \Ff\to \Ff$.
Recall that the category of elements $\el_{\KO(\Xx)}(\Ff)$ is a \emph{right distributive skew lattice} by Proposition~\ref{prp:Tskew}. This key property allows us to simplify notation with respect to Section~\ref{sct:patchmonad}. Our aim is to construct an explicit functor $l:\Sh(\Xx)^\PM \to\Sh_s(\Xx^p)$ quasi-inverse to the comparison functor $k:\Sh_s(\Xx^p)\to\Sh(\Xx)^\PM$, cf. Propositions~\ref{prp:Gsheaf} and~\ref{prp:Gsheaf2} below. This proves the main part of Theorem~\ref{thm:main2} with a minimal amount of point-set topology. %\marginpar{\tiny CB: is this ok ?}

The Boolean lattice $\KO(\Xx^p)$ is the Boolean envelope $\KO(\Xx)_B$ of the dual lattice $\KO(\Xx)$ and consists of the constructible subsets of $\Xx$, resp. the clopen subsets of $\Xx^p$, cf. Section \ref{subsct:patch}. Our task is thus to extend the sheaf $\Ff$ on $\KO(\Xx)$ to constructible subsets of $\Xx$ by using the skew distributive lattice structure of $(S_\Ff,\cm,\cj)=(\el_{\KO(\Xx)}(\Ff),\cm,\cj)$. To this purpose we introduce, for each constructible subset $W$ of $\Xx$, an equivalence relation $\sim_W$ on local sections $s',s''$ of $\Ff$ over compact open subsets $U',U''\supset W$, expressing that $s'$ and $s''$ agree on $W$ compatibly with the skew join $\cj$.%\marginpar{\tiny CB: is this ok ?}

Let $V\subset U$ be a pair of compact open subsets. We say that $(s',U),(s'',U)$ \emph{agree away from} $V$ if for all $(t,V)\in S_\Ff$ we have $t\cj s'=t\cj s''$. It amounts to the same to require that $s',s''$ are equalised by the quotient map $\Ff(U)\to \Ff_{U\backslash V}(U)$ of Proposition~\ref{prp:skewsheaf}iii. We say that $(s',U'),(s'',U'')$ are \emph{$W$-equivalent} if there exists $U\in\KO(\Xx)$ such that $W\subset U\subset U'\cap U''$ and for all $V$ with $W\subset U\backslash V$, the restrictions $s'_{|U}$ and $s''_{|U}$ agree away from $V$. This is clearly an equivalence relation which shall be denoted by $(s',U')\sim_W(s'',U'')$.

\begin{dfn}\label{dfn:constructive}For each clopen subset $W$ of $\,\Xx^p$ we set$$(l\Ff)(W)=\left(\coprod_{U\in\KO(\Xx),U\supset W}\Ff(U)\right)/\sim_W.$$\end{dfn}

\begin{prp}\label{prp:Gsheaf}$l\Ff$ defines a sheaf on $\KO(\Xx^p)$ such that $\Ff\cong kl\Ff$ as $T$-algebras. In particular, the stalks $(l\Ff)_x, x\in\Xx^p,$ are canonically isomorphic to the original stalks $\Ff_x,x\in\Xx,$ so that the support of $\,l\Ff$ is an upset.\end{prp}

\begin{proof}For $W'\subset W$ the $W$-equivalence relation refines the $W'$-equivalence relation so that $l\Ff$ is a presheaf on $\KO(\Xx^p)$. Since $\KO(\Xx^p)$ is a Boolean lattice, in order to see that $l\Ff$ is a sheaf, it suffices to show that for $W,W_1,W_2\in\KO(\Xx^p)$ with $W=W_1\cup W_2$ and $W_1\cap W_2=\emptyset$, we have a bijection$$|^W_{W_1}\times|^W_{W_2}:(l\Ff)(W)\overset{\cong}{\longrightarrow}(l\Ff)(W_1)\times(l\Ff)(W_2).$$

For the injectivity, let $(s_1,U_1),(s_2,U_2)\in S_\Ff$ such that $W\subset U_1\cap U_2$, and suppose that $(s_1,U_1)\sim_{W_1}(s_2,U_2)$ and $(s_1,U_1)\sim_{W_2}(s_2,U_2)$. Without restricting generality we can assume $U_1=U_2=U$. Let $V\subset U$ be a compact open subset avoiding $W$, and hence avoiding $W_1$ and $W_2$. By definition of the $W_i$-equivalence relations, $s_1$ and $s_2$ agree away from $V$ so that they are $W$-equivalent, proving injectivity.

For the surjectivity, let $(s_1,U_1),(s_2,U_2)\in S_\Ff$ such that $W_1\subset U_1$ and $W_2\subset U_2$. Any constructible subset of $\Xx$ is a disjoint union of finitely many differences $U-V$ for compact open subsets $V\subset U$ of $\Xx$. We can thus inductively assume that $W_1=U_1-V_1$ and $W_2=U_2-V_2$.  Let us now first suppose that ${s_1}_{|U_1\cap U_2}$ and ${s_2}_{|U_1\cap U_2}$ agree away from $V_1\cap V_2$. Then $(s_1\cj s_2,U_1\cup U_2)$ defines a $W$-equivalence class restricting to the $W_1$-equivalence class of $(s_1,U_1)$ and to the $W_2$-equivalence class of $(s_2,U_2)$, because $s_1\cj s_2$ restricts to $s_1$ on $U_1$ and to a local section on $U_2$ which agrees with $s_2$ away from $V_2$.

The general case reduces to this special case by putting $\tilde{s}_1=(s_2)_{|V_1\cap U_2}\cj s_1$ and $\tilde{s}_2=(s_1)_{|V_2\cap U_1}\cj s_2$. Indeed, $\tilde{s}_1$ and $\tilde{s}_2$ agree then on $U_1\cap U_2$ away from $V_1\cap V_2$, and we have $(s_1,U_1)\sim_{W_1}(\tilde{s}_1,U_1)$ and $(s_2,U_2)\sim_{W_2}(\tilde{s}_2,U_2)$.%\marginpar{\tiny CB: is this ok ?}

We get thus a sheaf $l\Ff$ on $\Xx^p$ and, by application of $k:\Sh(\Xx^p)\to\Sh(\Xx)^T$, a $T$-algebra $kl\Ff$ on $\Xx$. The unit of the $(\phi^*,\phi_*)$-adjunction yields a map of underlying sheaves $\Ff\to kl\Ff$ taking a local section $(s,U)$ of $\Ff$ to its $U$-equivalence class in $l\Ff$. Since for a compact open subset $U$, the $U$-equivalence relation is discrete, this is an isomorphism of sheaves on $\Xx$. In particular, we get the asserted isomorphim between the stalks of $\Ff$ and the stalks of $\,l\Ff$. Since the support of a spectral sheaf is automatically an upset, the same is true for the support of $l\Ff$.

It remains to be shown that $\Ff\cong kl\Ff$ is an isomorphism of $T$-algebras, i.e. the skew distributive lattices $S_\Ff$ and $S_{kl\Ff}$ are isomorphic, cf. Theorem~\ref{thm:main1}. For this, it suffices to observe that for any sheaf $\Gg$ on $\Xx^p$, and any compact open subsets $U,V$ of $\Xx$, the sheaf structure of $\Gg$ permits to define for local sections $(s,U),(t,V)$ of $\Gg$ a local section $(s\cj t, U\cup V)$ by glueing $s_{|U\backslash V}$ and $t$. This defines the join operation of the skew distributive lattice $S_{k\Gg}$. In the special case $\Gg=l\Ff$ we get the required isomorphism of skew distributive lattices $S_\Ff\cong S_{kl\Ff}$. %\marginpar{\tiny CB: is this ok ?}
\end{proof}

\begin{prp}\label{prp:Gsheaf2}For any sheaf $\,\Gg$ on $\Xx^p$ with saturated support, the canonical map $lk\Gg\to\Gg$ is an isomorphism.\end{prp}

\begin{proof}The $\PM$-algebra $k\Gg$ is by definition the sheaf $\phi_*\Gg$ equipped with the $T$-algebra structure $\phi_*\epsilon_\Gg$. The latter is given by

\begin{align*}
\phi_*\epsilon_\Gg\ \ \colon\ \ \PM \phi_*\Gg\ \ & \longrightarrow \ \ \phi_*\Gg\\
\bigveedot(s_i,a_i;b_i)\ & \mapsto\ \ \bigvee(s_i|^{a_i}_{b_i})\ \text{ (join taken in $\Gg$)}
\end{align*}
see Section~\ref{sct:patchmonad} for notation. From this $\PM$-algebra $k\Gg$ we obtain an associated sheaf
\begin{align*}
lk\Gg\ \ \colon\ \ \KO(\Xx^p)^\op\ \ & \longrightarrow \ \ \Set\\
W\ & \mapsto\ \ \{(s,U)\mid W\subset U\text{ and }s\in \Gg(U) \}/\sim_W
\end{align*}
as in Definition~\ref{dfn:constructive}, where $W$ is a clopen subset and $U$ a clopen upset of $\Xx^p$. We claim that $lkG\cong \Gg$. For each clopen subset $W$ consider the map
\[
(lk\Gg)(W)=\{(s,U)\mid W\subset U\text{ and } s\in \Gg(U) \}\longrightarrow  \Gg(W): (s,U)\mapsto s|^U_W
\]
and note that $(s_1,U_1)\sim_W(s_2,U_2)$ if and only if there is a clopen upset $U$ such that $W\subset U\subset U_1\cap U_2$ and the restrictions of $s_1,s_2$ to $U$ agree away from any $V\subset U$ avoiding $W$. This means that for any section $t$ of $\Gg$ over $V$
\[
t\cj s_1|^{U_1}_U=t\cj s_2|^{U_2}_U
\]
in $\Gg(U)$. Restricting both sides to $W$ we obtain (by Proposition~\ref{prp:skewsheaf}i) the identity
\[
s_1|^{U_1}_W=s_2|^{U_2}_W
\]
and this is clearly also sufficient. Therefore we obtain the equivalence
\[
(s_1,U_1)\sim_W(s_2,U_2)\quad\iff\quad s_1|^{U_1}_W=s_2|^{U_2}_W.
\]
yielding a bijection $lk\Gg(W)\cong \Gg(W)$ for all those clopen subsets $W$ for which either $\Gg(W)=\emptyset$ and $\Gg(U)=\emptyset$ if $U\supset W$, or then $\Gg(W)\not=\emptyset$ and $\Gg(U)\not=\emptyset$ if $U\supset W$. This is always the case if the support of $\Gg$ is an upset, cf. Lemma~\ref{lma:g=g}.\end{proof}

\begin{prp}\label{prp:Bool2}The category of right Boolean skew lattices (aka skew Boolean algebras) is a reflective subcategory of the category of right distributive skew lattices.\end{prp}

\begin{proof}We have seen in Theorem \ref{thm:main1} that the category of right distributive skew lattices over $L$ is equivalent to the category of global $\PM$-algebras over $\Spec(L)$. By Theorem \ref{thm:main2} the latter is equivalent to the category of global sheaves on $\Spec(L_B)$, which in turn by Proposition \ref{prp:Boolean} and Theorem \ref{thm:main0} is equivalent to the category of right Boolean skew lattices over $L_B$. Putting these equivalences together we can embed each right distributive skew lattice $S$ into a right Boolean skew lattice $S_B$ such that the following diagram$$\xymatrix{S\ar[r]\ar[d]&S_B\ar[d]\\S/\Dd\ar[r]&(S/\Dd)_B}$$ is a pullback diagram of right distributive skew lattices. The upper horizontal map corresponds to the inclusion of the sheaf $\Ff$ dual to $S$ into the sheaf $l\Ff$ dual to $S_B$, see Definition \ref{dfn:constructive} for an explicit formula of the Priestley sheaf $l\Ff$ and hence of the right Boolean skew lattice $S_B$ of local sections of $l\Ff$. Since the lower horizontal map is dual to patch refinement $\Spec(L_B)\to\Spec(L)$, the pullback property corresponds to the isomorphism $\Ff\cong kl\Ff$ of Proposition \ref{prp:Gsheaf}. The category of such pullback squares is thus equivalent to the category of right distributive skew lattices and contains the category of right Boolean skew lattices as the subcategory of those pullback squares in which the horizontal maps are invertible (i.e. $S$ is Boolean). It is clear that there is a reflection functor left adjoint to this inclusion.\end{proof}

\begin{rmk}The reflection functors of Propositions \ref{prp:Bool1} and \ref{prp:Bool2} are quite different. In the proof of Proposition \ref{prp:Bool1}, to each right distributive band $S$ over a distributive lattice $L$, is associated a right Boolean band $S'_B$ over the Boolean envelope $L_B$. The dual sheaf of $S'_B$ (cf. Theorem \ref{thm:main0}) may be identified with the inverse image $\phi^*\Ff$ of the dual sheaf $\Ff$ of $S$ with respect to $\phi:\Spec(L_B)\to\Spec(L)$. In particular, the sheaf $\PM\Ff=\phi_*\phi^*\Ff$ over $\Spec(L)$ is dual to the right distributive skew lattice obtained by pulling back $S'_B$ along the canonical inclusion $L\into L_B$.

If $S$ already underlies a right distributive skew lattice (cf. Corollary \ref{cor:forgetful}) the sheaf $\Ff$ dual to $S$ is a $T$-algebra by Theorem \ref{thm:main1}. Therefore, applying the functor $l$ of Definition \ref{dfn:constructive} to the action map $T\Ff\to\Ff$ yields an epimorphism $\phi^*F\to l\Ff$ inducing an epimorphism of right Boolean skew lattices $S'_B\to S_B$ over $L_B$. This relates the two reflection functors in an explicit manner.\end{rmk}

\
%%%%%%%%%%%%%%%%%%%%

\end{document}